\documentclass{amsart}
\usepackage{graphicx}
\usepackage{amssymb,amscd,amsthm,amsxtra}
\usepackage{latexsym}
\usepackage{epsfig}
\usepackage[osf]{libertine}
\usepackage{color}
\newcommand{\sus}[1]{{#1}}

\newtheorem{thm}{Theorem}[section]
\newtheorem{cor}[thm]{Corollary}
\newtheorem{lem}[thm]{Lemma}
\newtheorem{prop}[thm]{Proposition}
\theoremstyle{definition}
\newtheorem{defn}[thm]{Definition}
\theoremstyle{remark}
\newtheorem{rem}[thm]{Remark}
\numberwithin{equation}{section}
\newcommand{\R}{\mathbb R}

\newcommand{\eps}{\epsilon}

\newcommand{\p}{\partial}
\newcommand{\di}{\displaystyle}

\newcommand{\comment}[1]{}

\newcommand{\classG} {\mathcal{G}}
\def\C {\mathcal{C}}

\def\S {\mathbb{S}}

\def\bv {\mathbf{v}}
\def\bu {\mathbf{u}}

\def\eps{\varepsilon}

\def\Hh{\mathcal{H}}
\def\Ss{\mathcal{S}}

\def\Ceh{\mathcal{C}}
\def\Mah{\mathfrak{M}}
\def\Heh{\mathcal{H}}
\def\Seh{\mathcal{S}}
\def\Neh{\mathcal{N}}

\def\Feh{\mathcal{F}}
\def\Geh{\mathcal{G}}

\def\Beh{\mathcal{B}}
\newcommand{\loc}{\mathrm{loc}}

\DeclareMathOperator*{\tsum}{\textstyle{\sum}}

\newcommand{\de}[1] {\mathrm{d} #1}

\begin{document}
\title[Segregated configurations involving the square root of the laplacian]{Segregated configurations involving the square root of the laplacian and their free boundaries}
\date{\today}

\author[D. De Silva]{Daniela De Silva}\thanks{}
\address{Daniela De Silva \newline \indent
Barnard College, Columbia University  \newline \indent
Department of Mathematics,
 \newline \indent
 New York, NY 10027, USA}
\email{desilva@math.columbia.edu}

\author[S. Terracini]{Susanna Terracini}\thanks{}
\address{Susanna Terracini \newline \indent
 Dipartimento di Matematica ``Giuseppe Peano'', Universit\`a di Torino, \newline \indent
Via Carlo Alberto, 10,
10123 Torino, Italy}
\email{susanna.terracini@unito.it}

\date{\today} %%  this cancels date in article format

\subjclass[2010] {
35J70, % degenerate elliptic equations
35J75,  % singular elliptic equations
35R11, % Fractional partial differential equations
35B40, % Asymptotic behavior of solutions
35B44, % Blow-up
35B53, % Liouville theorems
%35K67, % Singular parabolic equations
}

\keywords{Free boundary problems, nonlocal diffusion, improvement of flatness, monotonicity formulas, blow-up classification}
%
%
%
%
%
%
%%%%%%%%%%%%%%%%%%%%%%%%%%%%%%%%%%%%%%%%%%%%%%%%%%%%%%%%%%%%%%%%%%%%%%%%%%%%%%%%%%%%%%%%%%%%%%%%%%%%%%%%
%%%%%%%%%%%%%%%%%%%%%%%%%%%%%%%%%%%%%%%%%%%%%%%%%%%%%%%%%%%%%%%%%%%%%%%%%%%%%%%%%%%%%%%%%%%%%%%%%%%%%%%%
%
%
%
%
%
%%%%%%%%%%%%%%%%%%%%%%%%%%%%%%%%%%%%%%%%%%%%%%%%%%%%%%%%%%%%%%%%%%%%%%%%%%%%%%%%%%%%%%%%%%%%%%%%%%%%%%%%

\thanks{Work partially supported by the ERC Advanced Grant 2013 n.~339958 Complex Patterns for Strongly Interacting Dynamical Systems - COMPAT }

\begin{abstract}

We study the local structure and the regularity of free boundaries 
of segregated minimal configurations involving the square root of the laplacian. We develop an improvement of flatness theory and, as a consequence of this and Almgren's monotonicity formula, we obtain partial regularity (up to a small dimensional set) of the nodal set, thus extending the known results in \cite{CL2008,TT} for the standard diffusion to some anomalous case.

\end{abstract}

\maketitle

\section{Introduction}

The analysis of the nodal sets of  segregated stationary configurations for systems of elliptic equations has been the subject of an intense study in the last decade, starting from the works \cite{ctvNehari,ctvOptimal,ctvVariational,CCCL2004,CL2008,CKL2009,nttv,dwz1,dwz2,dwz3}. The present paper is concerned with the geometric structure of the nodal set, when the creation of a free boundary is triggered by the interplay between \emph{fractional diffusion} and \emph{competitive interaction}. A prototypical example comes the following system of fractional Gross-Pitaevskii equations

\begin{equation}\label{eq:general_system}
\begin{cases}
(-\Delta + m_i^2)^{1/2}u_i+V(x)u_i=\omega_iu_i^3-\beta u_i\sum_{j\neq i} a_{ij}u^2_j\;\\
u_i\in H^{1/2}(\R^N),
\end{cases}
\end{equation}

\

\noindent with $a_{ij}=a_{ji}>0$, which is the relativistic version of the  Hartree-Fock approximation theory for mixtures of Bose-Einstein condensates in different hyperfine states which overlap in space.  The sign of $\omega_i$ reflects the type of interaction of the particles within each single state. If $\omega_i$ is positive, the self interaction is attractive (focusing problems defocusing otherwise). $V$ represents an external potential. The sign of $\beta$, on the other hand, accounts for the interaction of particles in different states. This interaction is attractive when negative and repulsive otherwise. If the condensates repel, and the competition rate tends to infinity, the  densities eventually separate spatially, giving rise to a free boundary: the common nodal set of the components $u_i$'s. This phenomenon is called phase separation and has been described in the recent literature, both physical and mathematical, in the case of standard diffusion. It is by now a well-established fact that in the case of elliptic systems with standard diffusion this nodal set is comparable, as regards to the qualitative properties, to that of the scalar solutions. The main reason can be attributed to the validity of a weak reflection law (see \cite{TT}) which constitutes the condition of extremality at the common interface.  Relevant connections have been established with optimal partition problems involving spectral functionals (cfr \cite{TT,RTT}).

We consider the following model with fractional diffusion: according to \cite{cs},  the $n$-dimensional half laplacian can be interpreted as a (nonlinear) Dirichlet-to-Neumann operator. For this reason we shall state all our results for harmonic functions with nonlinear Neumann boundary conditions involving strong competition terms.  Precisely, the following uniform-in-$\beta$ estimates have been derived in \cite{TVZ}. Here we denote by $B^+_r:= B_r\cap\{z>0\} \subset \R^{n+1}$ and by $\mathcal B_r:=B_r \cap\{z=0\}$ (where $z$ is the $(n+1)$-th coordinate). 
\begin{thm}[Local uniform H\"older bounds,  \cite{TVZ}]\label{thm: intro_local}
Let the functions $f_{i,\beta}$ be continuous and uniformly bounded (w.r.t. $\beta$) on bounded sets, and let $\{\bu_{\beta}=(u_{i,\beta})_{1\leq i\leq k}\}_{\beta}$ be a family of $H^1(B^+_1)$ solutions to the problems
\[
    \begin{cases}
    - \Delta u_i = 0 & \text{in } B^+_1\\
    \partial_{\nu} u_i = f_{i,\beta}(u_i) - \beta u_i \tsum_{j \neq i} u_j^2 & \text{on } \mathcal B_1.
    \end{cases} \tag{$P_{\beta}$}
\]%end{equation}
Let us assume that
\[
    \| \bu_{\beta} \|_{L^{\infty}(B^+_1)} \leq M,
\]
for a constant $M$ independent of $\beta$. Then for every $\alpha \in (0,1/2)$ there exists a constant
$C = C(M,\alpha)$, not depending on $\beta$, such that
\[
    \| \bu_\beta\|_{\C^{0,\alpha}\left(\overline{B^+_{1/2}}\right)} \leq C(M,\alpha).
\]
Furthermore, $\{\bu_{\beta}\}_{\beta }$ is relatively compact in $H^1(B^+_{1/2}) \cap \C^{0,\alpha}\left(\overline{B^+_{1/2}}\right)$ for every $\alpha < 1/2$.
\end{thm}

As a byproduct, up to subsequences, we have convergence of the solutions to ($P_{\beta}$) to some limiting profile, whose components are segregated on the boundary $\mathcal {B}_1$. If furthermore $f_{i,\beta}\to f_i$, uniformly on compact sets, we can prove that this limiting profile satisfies
\begin{equation}\label{eq:limiting_profiles}
    \begin{cases}
    - \Delta u_i = 0 & \text{in } B^+_1\\
    u_i \partial_{\nu} u_i = f_{i}(u_i)u_i & \text{on } \mathcal {B}_1.
    \end{cases}
\end{equation}
One can see that, for solutions of this type of equation, the highest possible regularity corresponds indeed to the H\"older exponent $\alpha=1/2$. As a matter of fact, it has been proved that the limiting profiles do enjoy such optimal regularity.
\begin{thm}[Optimal regularity of limiting profiles,  \cite{TVZ}]\label{thm: intro_limiting_prof}
Under the assumptions above, assume moreover that the locally Lipschitz continuous functions $f_i$ satisfy $f_i(s) = f_i'(0)s + O(|s|^{1+\eps})$ as $s\to 0$, for some $\eps>0$. Then $\bu\in\C^{0,1/2}\left(\overline{B^+_{1/2}}\right)$.
\end{thm}

\sus{It is worthwhile noticing that these result apply to the (local) minimizers of the functionals
\begin{equation}\label{eq:penalized_functional}
J_\beta(U)=\sum_{i=1}^k \int_{B^+_1} \frac12\nabla
u_i(x,z)|^2 dx\,dz +\beta\sum_{1\le i<j\le k} \int_{\mathcal {B}_1}  u_i^2 (x,0)u_j^2(x,0)\, dx
\end{equation}
in the set  of all configurations with fixed boundary data. Taking the singular limit as
\(\beta\to+\infty\) we are naturally lead to consider the energy minimizing profiles which segregate only at the characteristic hyperplane $\{z=0\}$.
}
\medskip

Our main goal is to describe, from  differential and geometric measure theoretical points of view, the structure of the trace on the characteristic hyperplane $\{z=0\}$ of the common nodal set of these limiting profiles. From now on, for the sake of simplicity we shall assume that the reactions $f_i$'s are identically zero and we reflect the components $u_i$'s through the hyperplane $\{z=0\}$. 
It is worthwhile noticing that we cannot deduce from the system \eqref{eq:limiting_profiles} alone any regularity property of the common nodal set $\Neh(\bu)=\{x\in \Omega\cap\R^n\times\{0\}:\bu(x,0)=0\}$, as the equations can be independently solved for arbitrary, though mutually disjoint,  nodal sets $\Neh(u_i)=\{(x,0)\in\Omega\cap\R^n\times\{0\}:u_i(x,0)=0\}$ on the characteristic hyperplane $\{z=0\}$. 

\sus{
\begin{defn}[Segregated minimal configurations, class $\Mah(\Omega)$]\label{def:segr_min_conf}
For an open, $z$-symme\-tric $\Omega\subset\R^{n+1}$, we define  the class $\Mah(\Omega)$ of the \emph{segregated minimal configurations} as the set of all the even-in-$z$ vector valued functions $\bu=(u_1,\ldots,u_k)\in (H^1(\Omega))^k\cap \Ceh^{0,1/2}(\Omega)$, whose components are all nonnegative and achieving the minimal
 the energy among configurations segregating only at the characteristic hyperplane $\{z=0\}$, that is solutions to 
\[
\min\left\{\sum_{i=1}^k \int_{\Omega}|\nabla u_i|^2\;:\;
\begin{array}{ll}
u_i(x,0)\cdot u_j(x,0)\equiv 0 \; \R^n\textrm{-a.e.}\;{\rm for }\;\; i\neq j,\\
u_i=\varphi_i, {\rm on}\;\partial\Omega\;\textrm{for }\;\; i=1,\dots,k,
\end{array}
\right\}
\]
where the $\varphi_i$'s are nonnegative $H^{1/2}$-boundary data which are even-in-$z$ and segregated on the hyperplane $\{z=0\}$. 
\end{defn}
}

For such class of solutions, we are going to prove a theorem on the structure of the nodal set $\Neh(\bu)$, which is the perfect counterpart of the results in \cite{CL2008,TT}.

\begin{thm}[Structure of the nodal set of segregated minimal configurations]\label{teo:main_structure_nodal}
Let $\Omega\subset \R^{n+1}$, with $n\geq2$, $\bu$ be a segregated  minimal configuration and let $\Neh(\bu)=\{x\in \tilde\Omega:\ \bu(x,0)=0\}$. Then, $\Neh(\bu)$ is the union of a relatively open \emph{regular part} $\Sigma_\bu$ and a relatively closed \emph{singular part} $\Neh(\bu)\setminus\Sigma_{\bu}$ with the following properties:
\begin{enumerate}
\item $\Sigma_\bu$ is a locally finite collection of  hyper-surfaces of class $C^{1,\alpha}$ (for some $0<\alpha<1$). 
\item $\Hh_\text{dim}(\Neh(\bu)\setminus\Sigma_{\bu})\leq n-2$ for any $n\geq 2$. Moreover, for $n=2$, $\Neh(\bu)\setminus\Sigma_{\bu}$ is a locally finite set.
\end{enumerate}
\end{thm}

\begin{rem}
(a) In the light of the extension facts related to the half-laplacian, our theory applies, among others, to segregated minimizing configurations involving non local energies, like, for instance the solutions to the following problem (when $s=1/2$):

\[
\min\left\{\sum_{i=1}^k \int_{\R^{2n}}\dfrac{|u_i(x)-u_i(y)|^2}{|x-y|^{n+2s}}\;:\;
\begin{array}{ll}
u_i(x,0)\cdot u_j(x,0)\equiv 0 \; \textrm{a.e. in}\;\R^n\;{\rm for }\;\; i\neq j,\\
u_i\equiv \varphi_i, {\rm on}\;\R^n\setminus \tilde \Omega\;\textrm{for }\;\; i=1,\dots,k,
\end{array}
\right\}
\]
where $\varphi_i$ are nonnegative $H^{1/2}(\R^n)$ data which are segregated themselves.  In this regard, our results extend those of \cite{CL2008} to the fractional case, or, equivalently, to the case when the phase segregation takes place only on the characteristic hyperplane. \\
(b) It is worthwhile noticing that, in case of the standard diffusion, the nodal set of the segregated minimal configurations shares the same measure theoretical features with the nodal set of harmonic functions; this is not the case of the fractional diffusion; indeed, as shown in \cite{stt2018}, the stratified structure of the nodal set of $s$-harmonic functions is far more complex than that of the segregated minimal configurations. The asymptotics and properties of limiting profiles of competition diffusion systems with quadratic (Lotka-Volterra) mutual interactions have been investigated in  \cite{VZ2014}; as discussed there, the free boundary, in the Lotka-Volterra case, resembles the nodal set of $s$-harmonic functions with some important differences however, enlightened in that paper.\\
(c) \sus{The theory developed in this paper is suitable to extend  in order to cover the limiting cases, as $\beta\to\infty$, of the variational problems associated with \eqref{eq:general_system}.}
\end{rem}

In order to prove our main result,  we will actually consider two other different notions of solution. The first is the class of solutions in the variational sense - concept that we have expressed through the validity of the domain variation formula in Definition \ref{def:segr ent prof} - and a second notion of solution, expressed in terms of viscosity solutions in Definitions \ref{def:viscosity2com} (for two components) and \ref{def:viscosity_solution} (for $k$ components). As we will see in Theorem \ref{teo:viscosity}, the latter notion is weaker (although they are probably equivalent) but carries precious information on the regularity of the nodal set.  In particular, both notions encode a reflection rule about the free boundary (see e.g. Proposition \ref{prop:wrl}),  which will be the ultimate reason for the regularity of the nodal set.

A major achievement toward the proof of our structure theorem will be  an improvement of flatness argument for the case of two components, which was inspired by the work in \cite{DR}. As a byproduct, it yields the following local regularity result:  in two dimensions, let  $U(t,z)$ be defined in polar coordinates as \begin{equation}\label{U}U(t,z) = r^{1/2}\cos \frac \theta 2, \end{equation}
$$t= r\cos\theta, \quad  z=r\sin\theta, \quad r\geq 0, \quad -\pi \leq  \theta \leq \pi,$$ and let
$$\bar U(t,z):= U(-t,z).$$ By abuse of notation we denote,
$$U(x,z):= U(x_n, z), \quad \bar U(x,z):= \bar U(x_n, z).$$

 \begin{thm} [Local regularity of the free boundary]\label{mainT}There exists $\bar \eps >0$ small depending only on $n$, such that if $\bu=(u_1,u_2)$ is a viscosity solution in $B_1$ in the sense of  Definition \ref{def:viscosity2com} satisfying
\begin{equation}\label{flat_intro} \|u_1 - U\|_\infty \leq \bar \eps, \quad \|u_2- \bar U\|_\infty \leq \bar \eps \end{equation} then $\Neh(\bu)$ is a $C^{1,\alpha}$ graph in $B_{\frac 1 2}\cap\R^n\times\{0\}$ for every $\alpha \in (0,1)$ with  $C^{1,\alpha}$ norm bounded by a constant depending on $\alpha$ and $n$. \end{thm}

The paper is organized as follows. Section \ref{sec:viscosity} contains the definition of viscosity solution for $2$-components systems, and the basic related facts. Section \ref{sec:linearized} is devoted to the study of the linearized problem associated with $\varepsilon$-domain variations around the fundamental solution of the system defined in \eqref{U}. Next, Section \ref{sec:harnack} contains Harnack estimates for viscosity solutions to the free boundary problem for two components. Such estimates will be crucial tools in proving the improvement of flatness result, in Section \ref{sec:improvement}, which concludes the analysis of the regular part of the free boundary of $2$-systems and proves Theorem \ref{mainT}. The rest of the paper concerns $k$-vector systems: Section \ref{sec:segregated} is devoted to the study of $k$-vector solutions in the variational sense (class $\Geh$) and the consequences of the associated Almgren's monotonicity formula, with a focus on the existence and classification of blow-ups and conic entire solutions. In Section \ref{sec:viscosity_k} we introduce the notion of $k$-vector solutions in the viscosity sense and we connect it to the variational one, in order to prove Theorem \ref{teo:main_structure_nodal}. The two appendices contain some ancillary known results.

\tableofcontents

\section{Viscosity solutions for two component systems}\label{sec:viscosity}

\subsection{Notations and definitions.} First, we introduce some notations. 

A point $X \in \R^{n+1}$ will be denoted by $X= (x,z) \in \R^n \times \R$, and sometimes $x=(x',x_n)$ with $x'=(x_1,\ldots, x_{n-1}).$ 

A ball in $\R^{n+1}$ with radius $r$ and center $X$ is denoted by $B_r(X)$ and for simplicity $B_r = B_r(0)$. Also, for brevity, $\mathcal{B}_r$ denotes the $n$-dimensional ball $B_r \cap \{z=0\}$ (previously denoted with $\p^0B_r^+$). 

Let $v(X)$ be a continuous non-negative function in $B_1$. We associate with $v$ the following sets: \begin{align*}
& B_1^+(v) := B_1 \setminus \mathcal N(v), \quad \mathcal N(v):=\{(x,0) : v(x,0) = 0 \};\\
& \mathcal{B}_1^+(v):= B_1^+(v) \cap \mathcal{B}_1;\\
%$$B_1^+(v):=\{(x,y,0)\in B_1 : v(x,y,0)>0\},$$
& F(v) := \p_{\R^n} \mathcal {N}^o(v) \cap \mathcal B_1, \quad \mathcal N^o(v):=Int_{\R^n} (\mathcal N(v)).
\end{align*}  

We now introduce the definition of viscosity solutions for a problem with two components. Let $u_1(x,z), u_2(x,z)$ be non-negative continuous functions in the  ball $B_1 \subset \R^{n+1}= \R^n \times \R,$ which vanish on complementary subsets of $\R^{n} \times\{ 0\}$ and are even in the $z$ variable. We consider the following free boundary problem

\begin{equation}\label{FBintro}\begin{cases}
\Delta u_i = 0, \quad \textrm{in $B_1^+(u_i), i=1,2,$}\\
\mathcal B_1= \mathcal N(u_1) \cup \mathcal N(u_2), \quad \mathcal N^o(u_1) \cap \mathcal N^o(u_2) = \emptyset,\\
\dfrac{\p u_1}{\p \sqrt t}=\dfrac{\p u_2}{\p \sqrt t} , \quad \textrm{on  $F(u_1,u_2):= \p_{\R^n} \mathcal N^o(u_1) \cap \mathcal B_1= \p_{\R^n} \mathcal N^o(u_2) \cap \mathcal B_1,$}
\end{cases}\end{equation}
where \begin{equation}\label{nabla_U}
\dfrac{\p u_i}{\p \sqrt t}(x_0): = \di\lim_{t \rightarrow 0^+} \frac{u_i(x_0+t\nu_i(x_0),0)}{\sqrt t},  \quad x_0 \in F(u_1,u_2)\end{equation} with $\nu_i(x_0)$ the normal to $F(u_1,u_2)$ at $x_0$ pointing toward  $\{u_i(x,0)>0\}$.

\begin{defn}Given $g, v$ continuous in $B_1$, we say that $v$
touches $g$ by below (resp. above) at $X_0 \in B_1$ if $g(X_0)=
v(X_0),$ and
$$g(X) \geq v(X) \quad (\text{resp. $g(X) \leq
v(X)$}) \quad \text{in a neighborhood $O$ of $X_0$.}$$ If
this inequality is strict in $O \setminus \{X_0\}$, we say that
$v$ touches $g$ strictly by below (resp. above).
\end{defn}

\begin{defn}\label{defsub} We say that the ordered pair $(v_1,v_2)$ is a (strict) comparison subsolution to \eqref{FBintro} if the $v_i \in C(B_1)$ are  non-negative functions even-in-$z$ that satisfy
\begin{enumerate}
\item $v_1$ is $C^2$ and $\Delta v_1 \geq 0$, in $ B_1^+(v_1),$\\
\item $v_2$ is $C^2$ and $\Delta v_2 \leq 0,$ in $ B_1^+(v_2),$\\
\item $\mathcal B_1= \mathcal N(v_1) \cup \mathcal N(v_2), \quad \mathcal N^o(v_1) \cap \mathcal N^o(v_2)=\emptyset, \quad \mathcal N^o(v_i) \neq \emptyset;$\\
%\item If $X_0 \in F(v)$ is regular from the positive side, i.e. there exists a tangent ball $B$ (in $\R^n$) to $F(v)$ at $X_0$ all contained in  $B_1^+(v)$ then 
\item $F(v_1,v_2):=F(v_1)=F(v_2)$ is $C^2$ and if $x_0 \in F(v_1,v_2)$ we have
$$v_i (x_0+t\nu_i(x_0),0) = \alpha_i(x_0) \sqrt t + o(\sqrt t), \quad \textrm{as $t \rightarrow 0^+,$}$$ with $$\alpha_1(x_0) \geq \alpha_2(x_0),$$ where $\nu_i(x_0)$ denotes the unit normal at $x_0$ to $F(v_1,v_2)$ pointing toward $\{v_i(x,0)>0\};$\\
\item Either the $v_i$ are both not harmonic in $B_1^+(v_i)$ or $\alpha_1(x_0) >\alpha_2(x_0)$ at all $x_0 \in F(v_1,v_2).$
 %and $r = \sqrt{|(x-x_0) \cdot \nu(x_0)|^2+ s^2}.$ 
\end{enumerate} 
\end{defn}

Similarly the ordered pair $(w_1,w_2)$ is  a (strict) comparison supersolution if $(w_2,w_1)$ is  a (strict) comparison subsolution. 

\begin{defn}\label{def:viscosity2com}We say that $(u_1,u_2)$ is a viscosity supersolution to \eqref{FBintro} if $u_i \geq 0$ is a  continuous function in $B_1$ which is even-in-$z$ and it satisfies
\begin{enumerate} \item $\Delta u_i = 0$ \quad in $B_1^+(u_i)$;\\ 

\item  $\mathcal B_1= \mathcal N(u_1) \cup \mathcal N(u_2), \quad \mathcal N^o(u_1) \cap \mathcal N^o(u_2)=\emptyset;$\\

\item If $(v_1,v_2)$ is a (strict) comparison subsolution then $v_1$ and $v_2$ cannot touch $u_1$ and $u_2$ respectively by below and above  at a point $X_0 = (x_0,0)\in F(u_1,u_2):=\p \mathcal N^o(u_1) \cap \mathcal B_1. $\\

Respectively, we say that  $(u_1,u_2)$ is a viscosity subsolution to \eqref{FBintro} if the conditions above hold with (iii) replaced by

\item If $(w_1,w_2)$ is a (strict) comparison supersolution then $w_1$ and $w_2$ cannot touch $u_1$ and $u_2$ respectively by above and below  at a point $X_0 = (x_0,0)\in F(u_1,u_2). $
\end{enumerate}

We say that $(u_1,u_2)$ is a viscosity solution if it is both a super and a subsolution.
\end{defn}

\subsection{Comparison principle.} We now derive a basic comparison principle.

\begin{defn} Let $(v_1,v_2)$ be a comparison subsolution to \eqref{FBintro}. We say that $(v_1,v_2)$ is monotone in the $e_n$ direction whenever $v_1$ is monotone increasing and $v_2$ is monotone decreasing in the $e_n$ direction.
\end{defn}

\begin{lem}[Comparison principle]\label{comppri}Let $(u_1,u_2), (v_1^t, v_2^t) \in C(\overline{B}_1)$ be respectively a solution and a family of  comparison subsolutions to \eqref{FBintro}, $t \in [0,1]$. Assume that
\begin{enumerate}
\item $v_1^0 \leq u_1, \quad v_2^0 \geq u_2$ in $\overline{B}_1;$
\item $v_1^t \leq u_1, \quad v_2^t \geq u_2$ on $\p B_1$ for all $t \in [0,1];$
\item $v_1^t < u_1$ on $\mathcal{F}(v_1^t)$ which is the boundary in $\p B_1$ of the set $\p \mathcal{B}_1^+(v_1^t) \cap \p \mathcal{B}_1$, for all $t\in [0,1];$
\item $v_2^t > u_2$ on $\mathcal{F}(u_2)$ which is the boundary in $\p B_1$ of the set $\p \mathcal{B}_1^+(u_2) \cap \p \mathcal{B}_1$, for all $t\in [0,1];$
\item $v_i^t$ is continuous in $(x,t) \in \overline{B}_1 \times [0,1]$ and $\overline{\mathcal B_1^+(v_i^t)}$ is continuous in the Hausdorff metric.
\end{enumerate}
Then 
\begin{equation*} v_1^t \leq u_1,  \quad v_2^t \geq u_2, \quad \text{in $\overline{B}_1$, for all $t\in[0,1]$.}
\end{equation*}
\end{lem}
\begin{proof}
Let $$A:= \{t \in [0,1] : v^t_1 \leq u_1, v_2^t \geq u_2 \ \ \text{on $\overline{B}_1$} \}.$$ In view of (i) and (v) $A$ is closed and non-empty. Our claim will follow if we show that $A$ is open. 
Let $t_0 \in A$, then $v^{t_0}_1 \leq u_1, v^{t_0}_2\geq u_2$ on $\overline{B}_1$ and by the definition of viscosity solution
$$F(v_i^{t_0}) \cap F(u_i) = \emptyset, \quad i=1,2.$$ For $i=1$ we argue as follows. Together with (iii) the identity above implies that 
$${\mathcal{B}_1^+(v_1^{t_0})} \subset  \mathcal{B}^+_1(u_1), \quad F(v_1^{t_0}) \cup \mathcal{F}(v_1^{t_0})  \subset  \{x \in \overline{\mathcal{B}}_1 : u_1(x,0)>0\}.$$ By (v) this gives that for $t$ close to $t_0$
\begin{equation}\label{inclus}
{\mathcal{B}_1^+(v_1^{t})} \subset  \mathcal{B}^+_1(u_1), \quad F(v_1^{t}) \cup \mathcal{F}(v_1^{t})  \subset  \{x \in \overline{\mathcal{B}}_1 : u_1(x,0)>0\}.
\end{equation} 
Call $D:= B_1 \setminus (\mathcal{B}^0_1(v_1^t) \cup F(v_1^t)) .$
Combining \eqref{inclus} with assumption (ii) we obtain 
$$v_1^t \leq u_1 \quad \text{on $\p D,$}$$ and by the maximum principle the inequality holds also in $D$. Hence 
$$v_1^t \leq u_1 \quad \text{in $\overline{B}_1.$}$$ Similarly for $i=2$, combining (iv) and (v) we get that for $t$ close to $t_0$, 
$$v_2^t \geq u_2 \quad \text{in $\overline{B}_1.$}$$
Hence $t \in A$ which shows that $A$ is open.
\end{proof}

The corollary below is now a straightforward consequence of Lemma \ref{comppri}.
\begin{cor} \label{compmon}Let $(u_1,u_2)$ be a solution to \eqref{FBintro} and let $(v_1,v_2)$ be a  comparison subsolution to \eqref{FBintro} in $B_2$ which is strictly monotone in the $e_n$ direction in $B_2^+(v_i)$. Call
$$v^t_i(X):=v_i(X+ t e_n), \quad X \in B_1.$$
Assume that for $-1 \leq t_0 < t_1\leq 1$
$$v^{t_0}_1 \leq u_1, \quad v^{t_0}_2 \geq u_2\quad \text{in $\overline{B}_1,$}$$
$$v^{t_1}_1 \leq u_1 \quad  \text{on $\p B_1,$}  \quad v_1^{t_1} < u_1  \quad \text{on $\mathcal{F}(v_1^{t_1}),$}$$
and 
$$v^{t_1}_2 \geq u_2 \quad  \text{on $\p B_1,$}  \quad v_2^{t_1} > u_2  \quad \text{on $\mathcal{F}(u_2).$}$$
Then 
\begin{equation*} v^{t_1}_1 \leq u_1, \quad v^{t_1}_2 \geq u_2\quad \text{in $\overline{B}_1.$}
\end{equation*}

\end{cor}

\subsection{Renormalization.} 

The following result allows us to replace the flatness assumption \eqref{flat_intro} with the property that $u_1$ and $u_2$ are trapped between two nearby translates of $U$ and $\bar U$ respectively. Precisely, we have the following lemma.
 
\begin{lem}\label{norm*} Let $(u_1, u_2)$ be non negative continuous functions in $\bar B_2$ satisfying
$$\Delta u_i = 0, \quad \textrm{in $ B_2^+(u_i)$},$$
\begin{equation}\label{disjoint}\mathcal B_1= \mathcal N(u_1) \cup \mathcal N(u_2), \quad \mathcal N^o(u_1) \cap \mathcal N^o(u_2) = \emptyset,\end{equation}
and the flatness assumption
\begin{equation}\label{F} \|u_1 - U\|_\infty \leq \delta, \quad \|u_2- \bar U\|_\infty \leq \delta \end{equation}
with $\delta>0$ small universal.
Then in $B_1$, 
\begin{equation}\label{flatflat}U(X-\eps e_n) \leq u_1(X) \leq U(X+\eps e_n), \quad \bar U(X+\eps e_n) \leq u_2(X) \leq \bar U(X-\eps e_n),\end{equation} for some $\eps=K\delta,$ $K$ universal.
\end{lem}

Lemma \ref{norm*} follows immediately from Lemma \ref{hyp} in the Appendix. Indeed, \eqref{disjoint}-\eqref{F} guarantee that $u_i$ satisfies the assumption \eqref{flat_fb}.

\subsection{Domain Variations} 

We recall the definition of $\eps$-domain variation  corresponding to $U$ and some basic lemmas from \cite{DR}. We also introduce in a similar fashion the $\eps$-domain variation  corresponding to $\bar U$ and deduce its properties.

Let $\eps>0$ and let $g$ be a  continuous non-negative function in $\overline{B}_\rho$. 
Here and henceforth we denote by $P^\pm$ the half-hyperplanes $$P^+:= \{X \in \R^{n+1} : x_n \geq 0, z=0\},$$ $$P^-:= \{X \in \R^{n+1} : x_n \leq 0, z=0\},$$ and by $$L:= \{X \in \R^{n+1}: x_n=0, z=0\}.$$  
To each $X \in \R^{n+1} \setminus P^-$ we associate $\tilde g_\eps(X) \subset \R$ via the formula 
\begin{equation}\label{deftilde} U(X) = g(X - \eps w e_n), \quad \forall w \in \tilde g_\eps(X).\end{equation} 
Similarly, 
to each $X \in \R^{n+1} \setminus P^+$ we associate $\tilde {\bar g}_\eps(X) \subset \R$ via the formula 
\begin{equation}\label{deftilde2} \bar U(X) = g(X - \eps w e_n), \quad \forall w \in \tilde {\bar g}_\eps(X).\end{equation} 

By abuse of notation, we write $ \tilde g_\eps(X), \tilde{\bar g}_\eps(X)$ to denote any of the values in this set. 

As observed in \cite{DR}, if g satisfies\begin{equation}\label{flattilde}U(X - \eps e_n) \leq g(X) \leq U(X+\eps e_n) \quad \textrm{in $B_\rho,$}\end{equation} then for all $\eps >0$  we can associate with $g$ a possibly multi-valued function $\tilde{g}_\eps$ defined at least on $B_{\rho-\eps} \setminus P^-$ and taking values in $[-1,1]$ which satisfies \begin{equation} \label{til}U(X) = g(X - \eps \tilde{g}_\eps(X)e_n).\end{equation} Moreover if $g$ is strictly monotone  in the $e_n$ direction in $B^+_\rho(g)$, then $\tilde{g}_\eps$ is single-valued. A similar statement holds for $\tilde {\bar g}_\eps$, when $g$ satisfies the flatness assumption\begin{equation}\label{flattilde2}\bar U(X + \eps e_n) \leq g(X) \leq \bar U(X-\eps e_n) \quad \textrm{in $B_\rho.$}\end{equation}

The following elementary lemmas hold. The proof of the first one can be found in \cite{DR}. The second one can be obtained similarly.

\begin{lem}\label{elem} Let $g, v$ be non-negative continuous functions in $B_\rho$. Assume that $g$ satisfies the flatness condition \eqref{flattilde} in $B_\rho$ and that $v$ is strictly increasing in the $e_n$ direction in $B_\rho^+(v).$ Then if 
$$v \leq g \quad \text{in $B_\rho,$}$$ and 
$\tilde v_\eps$ is defined on $B_{\rho-\eps} \setminus P^-$ we have that
$$\tilde v_\eps \leq \tilde g_\eps \quad \text{on  $B_{\rho-\eps} \setminus P^-.$}$$
Viceversa, if $\tilde v_\eps$ is defined on $B_s \setminus P^-$ and 
$$\tilde v_\eps \leq \tilde g_\eps \quad \text{on $B_s \setminus P^-,$}$$ then
$$v \leq g \quad \text{on $B_{s-\eps}$}.$$
\end{lem}

\begin{lem}\label{elem2} Let $g, v$ be non-negative continuous functions in $B_\rho$. Assume that $g$ satisfies the flatness condition \eqref{flattilde2} in $B_\rho$ and that $v$ is strictly decreasing in the $e_n$ direction in $B_\rho^+(v).$ Then if 
$$v \geq g \quad \text{in $B_\rho,$}$$ and 
$\tilde {\bar v}_\eps$ is defined on $B_{\rho-\eps} \setminus P^+$ we have that
$$\tilde {\bar v}_\eps \leq \tilde {\bar g}_\eps \quad \text{on  $B_{\rho-\eps} \setminus P^+.$}$$
Viceversa, if $\tilde {\bar v}_\eps$ is defined on $B_s \setminus P^+$ and 
$$\tilde {\bar v}_\eps \leq \tilde {\bar g}_\eps \quad \text{on $B_s \setminus P^+,$}$$ then
$$v \geq g \quad \text{on $B_{s-\eps}$}.$$
\end{lem}

We now state and prove a key comparison principle, which will follow immediately from the lemmas above and Corollary  \ref{compmon}.

\begin{lem}\label{linearcomp}
Let $(u_1,u_2),  (v_1,v_2)$ be respectively a solution and a comparison subsolution to \eqref{FBintro} in $B_2$, with $(v_1,v_2)$ strictly monotone in the $e_n$ direction in $B_2^+(v_i).$ Assume that $u_1$ and $u_2$ satisfy respectively the flatness assumptions \eqref{flattilde}-\eqref{flattilde2} in $B_2$
for $\eps>0$ small  and that $\widetilde {(v_1)}_\eps, \widetilde{(\overline{v_2})}_\eps$ are defined in $B_{2-\eps} \setminus P^\mp$ and satisfy  $$|\widetilde {(v_1)}_\eps| \leq C, \quad |\widetilde {(\overline{v_2})}_\eps| \leq C.$$
If
\begin{equation}\label{start}
 \widetilde{{(v_1)}}_\eps + c \leq \widetilde {(u_1)}_\eps   \quad \text{in $(B_{3/2} \setminus \overline{B}_{1/2}) \setminus P^-,$} 
\end{equation} 
and
\begin{equation}\label{start2}
 \widetilde{(\overline{v_2})}_\eps + c \leq \widetilde {(\overline{u_2})}_\eps   \quad \text{in $(B_{3/2} \setminus \overline{B}_{1/2}) \setminus P^+,$} 
\end{equation} 
then 
\begin{equation}\label{conclusion}
 \widetilde{{(v_1)}}_\eps + c \leq \widetilde {(u_1)}_\eps  \quad \text{in $B_{3/2} \setminus P^-,$} 
\end{equation} 
and
\begin{equation}\label{conclusion2}
 \widetilde{(\overline{v_2})}_\eps + c \leq \widetilde {(\overline{u_2})}_\eps   \quad \text{in $B_{3/2}  \setminus P^+.$} 
\end{equation} 
\end{lem}
\begin{proof} We wish to apply Corollary \ref{compmon} to the functions $(u_1,u_2)$ and $$v_{i,t} = v_i(X+ \eps t e_n).$$

We need to verify that for some  $t_0 < t_1=c$
\begin{equation}\label{n1}v_{1,t_0} \leq u_1 \quad v_{2,t_0} \geq u_2  \quad \text{in $\overline{B}_1,$}\end{equation} and for all $\delta>0$ and small
\begin{equation}\label{n2}v_{1,t_1-\delta} \leq u_1 \quad  \text{on $\p B_1,$}  \quad v_{1,t_1-\delta} < u_1  \quad \text{on $\mathcal{F}(v_{1,t_1-\delta}),$}\end{equation}
\begin{equation}\label{n3}v_{2,t_1-\delta} \geq u_2 \quad  \text{on $\p B_1,$}  \quad v_{2,t_1-\delta} > u_2  \quad \text{on $\mathcal{F}(u_2).$}\end{equation}
Then our Corollary implies
\begin{equation*} v_{1,t_1-\delta} \leq u_1 \quad \text{in $\overline{B}_1$,} \end{equation*}
\begin{equation*} v_{2,t_1-\delta} \geq u_2 \quad \text{in $\overline{B}_1$.} \end{equation*}
By letting $\delta$ go to 0, we obtain that 
\begin{equation*} v_{1,t_1} \leq u_1, \quad v_{2,t_1} \geq u_2 \quad \text{in $\overline{B}_1$,} \end{equation*} which in view of Lemma \ref{elem} gives 
\begin{equation*}\widetilde{(v_{1,t_1})}_\eps  \leq \widetilde {(u_1)}_\eps \quad \text{in $B_{1-\eps}\setminus P^-$,}\end{equation*}
\begin{equation*}\widetilde{({\overline{ v_{2,t_1}}})}_\eps  \leq \widetilde {(\overline{u_2})}_\eps \quad \text{in $B_{1-\eps}\setminus P^+$,}\end{equation*} 
 assuming that the $\eps$-domain variations on the left hand side exist on $B_{1-\eps} \setminus P^\mp.$ On the other hand, it is easy to verify that on such set
\begin{equation}\label{tildetrans}\widetilde{(v_{1,t})}_\eps (X) = (\widetilde {v_1})_\eps (X) + t, \quad \widetilde{(\overline{v_{2,t}})}_\eps (X) = \widetilde {({\overline{v_2})}}_\eps (X) + t,\end{equation} and hence we have for $t_1=c/2,$
\begin{equation*}\widetilde {(v_1)}_\eps + t_1 \leq \widetilde {(v_1)}_\eps + c \leq \widetilde {(u_1)}_\eps \quad \text{in $B_{1-\eps} \setminus P^-,$} \end{equation*} 
\begin{equation*}\widetilde {(\overline{v_2})}_\eps + t_1 \leq \widetilde {(\overline{v_2})}_\eps + c \leq \widetilde {(\overline{u_2})}_\eps \quad \text{in $B_{1-\eps} \setminus P^+,$} \end{equation*} 
which gives the desired conclusion.

We are left with the proof of \eqref{n1}-\eqref{n2}-\eqref{n3}.

In view of Lemmas \ref{elem}-\ref{elem2}, in order to obtain \eqref{n1} it suffices to show that
$$\widetilde{(v_{1,t_0})}_\eps  \leq \widetilde {(u_1)}_\eps, \quad \text{in $B_{1+\eps} \setminus P^-,$}$$
$$\widetilde{(\overline{v_{2,t_0}})}_\eps  \leq \widetilde {(\overline{u_2})}_\eps, \quad \text{in $B_{1+\eps} \setminus P^+,$}$$
which by \eqref{tildetrans} becomes
$$\widetilde{(v_1)}_\eps + t_0 \leq \widetilde{(u_1)}_\eps, \quad \text{in $B_{1+\eps} \setminus P^-,$}$$
$$\widetilde {(\overline{v_2})}_\eps + t_0 \leq \widetilde {(\overline{u_2})}_\eps, \quad \text{in $B_{1+\eps} \setminus P^+.$}$$
These last inequalities hold trivially for an appropriate choice of $t_0$ since the functions involved are bounded.

For \eqref{n2}-\eqref{n3}, notice that the first inequality follows easily from our assumption \eqref{start}-\eqref{start2} together with \eqref{tildetrans} and Lemmas \ref{elem}-\ref{elem2}.  More precisely we have that for all $\delta \leq c/2,$ the stronger statement

\begin{equation}\label{new1}v_{1,t_1} \leq u_1 \quad v_{2,t_1+\delta} \geq u_2 \quad \text{in $B_{\frac 3 2-\eps} \setminus B_{\frac 1 2+\eps}$,}\end{equation} holds.

In particular, from the strict monotonicity of $v_1$ in the $e_n$-direction in $B_2^+(v_1)$ we have that
$$v_{1,t_1} > 0 \quad \text{on $\mathcal{F}(v_{1,t_1-\delta}),$}$$
which combined with the previous inequality gives that
$$u_1 >0 \quad \text{on $\mathcal{F}(v_{1,t_1-\delta}),$}$$
that is the second condition in \eqref{n2}. Similarly the second inequality in \eqref{n2} follows from the second inequality in \eqref{new1} and the strict monotonicity of $v_2$ in $B_2^+(v_2).$
\end{proof}

Finally, given $\eps>0$ small and Lipschitz functions $\tilde{\varphi}, \tilde{\bar \varphi}$ defined on $B_{\rho}(\bar X)$,  with values in $[-1,1]$, then there exists a unique pair $(\varphi_\eps, \bar \varphi_\eps)$ defined at least on $B_{\rho-\eps}(\bar X)$ such that \begin{equation} U(X) = \varphi_\eps(X - \eps \tilde{\varphi}(X)e_n), \quad X \in B_\rho(\bar X),\end{equation}
\begin{equation} \bar U(X) = \bar \varphi_\eps(X - \eps \tilde{\bar\varphi}(X)e_n), \quad X \in B_\rho(\bar X).\end{equation}
%Moreover such pair is increasing in the $e_n$-direction. 

It is readily seen that if
 $(u_1,u_2)$ satisfies the flatness assumption \eqref{flatflat} in $B_1$ then (say $\rho,\eps<1/4$, $\bar X \in B_{1/2},$)
\begin{equation}\label{gtildeg}\tilde \varphi \leq (\tilde u_1)_\eps \quad \text{in $B_\rho(\bar X) \setminus P^-$} \Rightarrow \varphi_\eps \leq u_1 \quad \text{in $B_{\rho -\eps}(\bar X)$},\end{equation}
\begin{equation}\label{gtildeg2}\tilde {\bar\varphi} \leq (\tilde u_2)_\eps \quad \text{in $B_\rho(\bar X) \setminus P^+$} \Rightarrow \bar\varphi_\eps \geq u_2 \quad \text{in $B_{\rho -\eps}(\bar X)$}.\end{equation}

The following proposition is contained in \cite{DS}. An analogous statement can be clearly obtained for ${\bar \varphi}_\eps.$
\begin{prop}\label{1}Let $\varphi$ be a smooth function in $B_\lambda(\bar X)\subset \R^{n+1} \setminus P^-$. Define (for $\eps>0$ small) the function $\varphi_\eps$ as above by\begin{equation}\label{deftilde3} U(X) = \varphi_\eps(X -  \eps\varphi(X)e_n).\end{equation}Then, \begin{equation}\label{laplaceest} \Delta \varphi_\eps = \eps \Delta(U_n \varphi)+ O(\eps^2),\quad \textrm{in $B_{\lambda/2}(\bar X)$}\end{equation} with the function in $O(\eps^2)$ depending on $\|\varphi\|_{C^5}$ and $\lambda$.\end{prop}

\section{The linearized problem.} \label{sec:linearized}

We introduce here the linearized problem associated with \eqref{FBintro}. 
Given  $g \in C(B_1)$  and  $X_0=(x'_0,0,0) \in B_1 \cap L,$ we call
$$g_r (X_0) := \di\lim_{(x_n,z)\rightarrow (0,0)} \frac{g(x'_0,x_n, z) - g(x'_0,0,0)}{r}, \quad   r^2=x_n^2+z^2 .$$
Once the change of unknowns  \eqref{deftilde}-\eqref{deftilde2} has been done, the linearized problem associated with \eqref{FBintro} is 
\begin{equation}\label{linear}\begin{cases} \Delta (U_n g_1) = 0, \quad \text{in $B_1 \setminus P^-,$}\\ 
 \Delta (-\bar U_n g_2) = 0, \quad \text{in $B_1 \setminus P^+,$}\\
g_1=g_2, \quad (g_1)_r+(g_2)_r=0, \quad \text{on $B_1\cap L$.}\end{cases}\end{equation}

\begin{defn}\label{linearsol}We say that $(g_1,g_2)$ is a viscosity solution to \eqref{linear}  if $g_i\in C(B_1)$, $g_i$ is even-in-$z$ and it satisfies
\begin{enumerate}\item $\Delta (U_n g_1) = 0$ \quad in $B_1 \setminus P^-$;\\ 
\item $\Delta (-\bar U_n g_2) = 0$ \quad in $B_1 \setminus P^+$;\\
\item $g_1=g_2 \quad \text{on $B_1 \cap L;$}$\\
\item Let $\phi_i$ be continuous around  $X_0=(x'_0,0,0) \in B_1 \cap L$ and satisfy $$\phi_1(X) =l(x') + b_1(X_0) r + O(|x'-x'_0|^2 + r^{3/2}), $$
$$\phi_2(X) = l(x')+ b_2(X_0) r + O(|x'-x'_0|^2+r^{3/2}), $$ with $l(x')= a_0 + a_1\cdot (x'-x'_0).$ 
 
\noindent If $b_1(X_0) +b_2(X_0)>0$ then  $(\phi_1, \phi_2)$ cannot touch $(g_1,g_2)$ by below at $X_0$,  and if $b_1(X_0)+ b_2(X_0)<0$ then $(\phi_1,\phi_2)$ cannot touch $(g_1,g_2)$ by above at $X_0$. \end{enumerate}\end{defn}

We wish to prove the following existence and regularity result.

\begin{thm}\label{class} Given a boundary data $(h_1,h_2)$ with  $h_i \in C(\bar B_1), |h_i| \leq 1$, $h_i$ even-in-$z$ and $h_1=h_2 $ on $B_1 \cap L$, there exists a unique classical solution $(g_1,g_2)$ to \eqref{linear} such that $g_i \in C(\overline{B}_1)$, $g_i = h_i$ on $\p B_1$, $g_i$ is even-in-$z$ and it satisfies
\begin{equation}\label{mainh}|g_i(X) - a_0 -  a' \cdot (x'- x'_0)- b_i r| \leq C (|x'-x'_0|^2 + r^{3/2}), \quad X_0 \in B_{1/2} \cap L,\end{equation} for a universal constants $C$, and $a' \in \R^{n-1}, a_0, b_i \in \R$ depending on $X_0$ and satisfying
$$b_1+b_2=0.$$
\end{thm}

As a corollary of the theorem above we obtain the following regularity result.

\begin{thm}[Improvement of flatness]\label{lineimpflat} There exists a universal constant $C$ such that if $(g_1,g_2)$ is a viscosity solution to \eqref{linear} in $B_1$ with $$-1 \leq g_i(X) \leq 1\quad \text{in $B_1,$}$$ then  \begin{equation}\label{boundlin2}  a_0+ a' \cdot x' + b_i r-C|X|^{3/2} \leq g_i(X) \leq a_0+a' \cdot x' + b_i r+ C |X|^{3/2},\end{equation}for some $a'\in \R^{n-1}, a_0, b_i \in \R$ such that
$$b_1+b_2=0.$$
\end{thm}
\begin{proof}
Let $(w_1,w_2)$ be the unique classical solution to \eqref{linear} in $B_{1/2}$ with boundary data $(g_1,g_2)$. We will prove that $w_i=g_i$ in $B_{1/2}$ and hence it satisfies  the desired estimate in view of \eqref{mainh}. Denote by $$\bar w^i_\eps := w_i - \eps + \eps^2 r.$$ Then, for $\eps$ small $$\bar w^i_\eps < g_i \quad \text{on $\p B_{1/2}$}.$$ We wish to prove that
\begin{equation}\label{Ub} \bar w^i_\eps \leq g_i \quad \text{in $B_{1/2}$.}\end{equation}
Now, notice that  $(\bar w^1_\eps,\bar w^2_\eps)$ (and all its translations) is a classical strict subsolution to \eqref{linear} that is \begin{equation}\label{linear3}\begin{cases} \Delta (U_n \bar w^1_\eps) = 0, \quad \text{in $B_{1/2} \setminus P^-,$}\\
\Delta (-\bar U_n \bar w^2_\eps) = 0, \quad \text{in $B_{1/2} \setminus P^+,$}\\ 
\bar w^1_\eps=\bar w^2_\eps, \quad  (\bar w^1_\eps)_r+(\bar w^2_\eps)_r>0, \quad \text{on $B_{1/2}\cap L$.}\end{cases}\end{equation}
Since $g_i$ is bounded, for $t$ large enough $\bar w^i_\eps - t$ lies strictly below $g_i$. We let $t \rightarrow 0$ and show that the first contact point cannot occur for $t \geq 0$. Indeed since $\bar w^i_\eps -t$ is a strict subsolution which is strictly below $g_i$ on $\p B_{1/2}$ then no touching can occur either in $B_{1/2} \setminus P^\mp$ or on $B_{1/2} \cap L.$ We only need to check that no touching occurs on $ P^\mp \setminus L.$ This follows from Lemma \ref{rk} in the Appendix.

Thus \eqref{Ub} holds. 
%Analogously set 
%$$\underline{h}_\eps = h + \eps -\eps^2 r,$$ we can show that 
%\begin{equation}\label{Lb} \underline{h}_\eps \geq w \quad \text{in $B_{1/2}$.} 
%\end{equation}
%Combining \eqref{Ub} and \eqref{Lb} we obtain that $$|w- h| \leq \eps \quad \text{in $B_{1/2}$}$$
%and letting $\eps \rightarrow 0$ we obtain $$w=h.$$ 
%The desired bound \eqref{boundlin2} follows from \eqref{mainh}.
Passing to the limit as $\eps \rightarrow 0$ we deduce that $$w_i \leq g_i \quad \text{in $B_{1/2}$.}$$ Similarly we also infer that $$w_i \geq g_i \quad \text{in $B_{1/2},$}$$ and the desired equality holds. 
\end{proof}

We remark that there is no theory readily available for this class of degenerate problems with a Neumann type boundary condition on the thin boundary $L$, which in this setting has positive capacity (see the classical work on degenerate elliptic problems \cite{FKSR}). Thus Theorem \ref{lineimpflat} calls for a direct proof.

The existence of the classical solution of Theorem \ref{class} will be achieved via a variational approach in the appropriate functional space. Precisely, 
define the weighted Sobolev space
$$H^1_\omega(B_1):= H^1(B_1, U_n^2 dX) \times H^1(B_1, \bar U_n^2 dX)$$ endowed with the norm 
$$\|(h_1,h_2)\|: =  \int_{B_1} (h_1^2 + |\nabla h_1|^2)U_n^2dX+  \int_{B_1}  (h_2^2 +  |\nabla h_2|^2)\bar U_n^2dX$$
and its subspace
$$V_0(B_1) := \{(h_1,h_2): h_i \in C_0^\infty(B_1), h_1=h_2 \quad \text{on $B_1 \cap L$}\}.$$
Let $\mathcal H^1_{0,\omega}(B_1):= \overline{V_0(B_1)}$ be the completion of $V_0(B_1)$ in $H_\omega^1(B_1).$

Consider the energy functional 
$$J(h_1,h_2) := \int_{B_1} (U_n^2 |\nabla h_1|^2 + \bar U_n^2 |\nabla h_2|^2)dX, \quad (h_1,h_2) \in H_\omega^1(B_1).$$
Given a boundary data $(h_1,h_2)$ with  $h_i \in C^\infty(\bar B_1),$ and $h_1=h_2$ on $\bar B_1 \cap L$, we minimize
$$J(h_1+\varphi_1, h_2+\varphi_2)$$ among all $(\varphi_1,\varphi_2) \in \mathcal H_{0,\omega}^1(B_1).$
Equivalently, a minimizing pair $(g_1,g_2)$ must satisfy
$$J(g_1,g_2) \leq J(g_1+\varphi_1, g_2+\varphi_2), \quad \forall \varphi_i \in C_0^\infty(B_1), i=1,2,$$ with 
\begin{equation}\label{uguali}\varphi_1=\varphi_2 \quad \text{on $B_1 \cap L.$}\end{equation}
Since $J$ is strictly convex, we conclude that $(g_1,g_2)$ is a minimizer if and only if
$$\lim_{\eps \to 0} \frac{J(g_1,g_2)-J(g_1+\eps \varphi_1, g_2+\eps \varphi_2)}{\eps}=0, \quad \forall \varphi_i \in C_0^\infty(B_1), i=1,2,$$ satisfying \eqref{uguali}, or equivalently, 
\begin{equation}\label{EL}\int_{B_1} (U_n^2 \nabla g_1 \cdot \nabla \varphi_1 + \bar U_n^2 \nabla g_2 \cdot \nabla \varphi_2) dx =0, \quad \forall \varphi_i \in C_0^\infty(B_1), i=1,2,\end{equation}
satisfying \eqref{uguali}.

\begin{lem}\label{Unhisharm} Let $(g_1,g_2)$ be a minimizer to $J$ in $B_1$, then $$\Delta (U_n g_1) = 0 \quad \text{in $B_1 \setminus P^-,$}$$
$$\Delta (-\bar U_n g_2) = 0 \quad \text{in $B_1 \setminus P^+.$}$$
\end{lem}
\begin{proof} 
By \eqref{EL}
$$\int_{B_1} (U_n^2 \nabla g_1 \cdot \nabla \varphi^- + \bar U_n^2 \nabla g_2 \cdot \nabla \varphi^+) dx =0, \quad \forall \varphi^\mp \in C_0^\infty(B_1\setminus P^\mp). $$
After integration by parts we obtain,
$$\text{div} (U_n^2 \nabla g_1)= 0 \quad \text{in $B_1 \setminus P^-,$}$$
$$\text{div} (\bar U_n^2 \nabla g_2)= 0 \quad \text{in $B_1 \setminus P^+$}.$$
Since $U_n, \bar U_n$ are bounded in $B_1 \setminus P^-, B_1\setminus P^+$ respectively, we conclude by elliptic regularity that
$$g_1 \in C^\infty(B_1 \setminus P^-), g_2 \in C^\infty(B_1 \setminus P^+).$$ Now a simple computation concludes the proof.  Precisely, $$\text{div} (U_n^2 \nabla g_1)=U_n^2 \Delta g_1 + 2\sum_{i=1}^{n+1} U_n U_{ni}(g_1)_i = 0 \quad \text{in $B_1 \setminus P^-.$}$$
Since $U_n > 0$ and $\Delta U=0$ in $B_1 \setminus P^-$ the identity above is equivalent to
$$\Delta(U_n g_1) = U_n \Delta g_1 + 2 \nabla U_n \cdot \nabla g_1=0 \quad \text{in $B_1 \setminus P^-,$}$$ as desired. The computation for $g_2$ is obtained analogously.
\end{proof}

\begin{lem}\label{char}Let $(g_1,g_2) \in C(B_1)$ be a solution to$$\Delta (U_n g_1) = 0 \quad \text{in $B_1 \setminus P^-,$}$$
$$\Delta (-\bar U_n g_2) = 0 \quad \text{in $B_1 \setminus P^+,$}$$
with $g_1=g_2$ on $B_1 \cap L$ and assume that 
$$\di\lim_{r \rightarrow 0} (g_i)_r(x',x_n,z) = b_i(x'),\quad (g_i)_r(X)= \frac{x_n}{r} g_{x_n}(X) + \frac{z}{r} g_z(X);$$
with $b_i(x')$  a continuous function. Then $(g_1,g_2)$ is a minimizer to $J$ in $B_1$ if and only if $b_1+b_2 \equiv 0.$\end{lem}
\begin{proof} The pair $(g_1,g_2)$ minimizes $J$ if and only if
$$\int_{B_1} (U_n^2 \nabla g_1 \cdot \nabla \varphi_1 + \bar U_n^2 \nabla g_2 \cdot \nabla \varphi_2) dx =0,$$ for all $\varphi_i \in C_0^\infty(B_1)$ which coincide on $L$.

We compute that
\begin{equation}\label{g1} \int_{B_1} U_n^2 \nabla g_1 \cdot \nabla \varphi_1dx = \int_{B_1 \setminus P^-} \text{div}(U_n^2 g_1)\varphi_1 dx + \lim_{\delta \to 0} \int_{\p C_\delta \cap B_1} U_n^2 \varphi_1 \nabla g_1 \cdot \nu d\sigma \end{equation} where $C_\delta= \{r \leq \delta\}$ and $\nu$ is the inward normal to $C_\delta,$ and similarly
\begin{equation}\label{g2}\int_{B_1} \bar U_n^2 \nabla g_2 \cdot \nabla \varphi_2dx = \int_{B_1 \setminus P^+} \text{div}(\bar U_n^2 g_2)\varphi_2 dx + \lim_{\delta \to 0} \int_{\p C_\delta \cap B_1} \bar U_n^2 \varphi_2 \nabla g_2 \cdot \nu d\sigma.\end{equation}
Indeed, say for $g_1,$ given $\delta>0$, for $P_\eps$ a strip of width $\eps$ around $P^-$
$$\int_{B_1\setminus C_\delta} U_n^2 \nabla g_1 \cdot \nabla \varphi_1dx=\lim_{\eps \to 0} \int_{(B_1\setminus C_\delta) \cap (B_1 \setminus P_\eps)} U_n^2 \nabla g_1 \cdot \nabla \varphi_1dx$$ and integrating by parts

$$\int_{B_1\setminus C_\delta} U_n^2 \nabla g_1 \cdot \nabla \varphi_1dx= \int_{(B_1\setminus C_\delta) \setminus P^-} \text{div}(U_n^2 g_1)\varphi_1 dx +  \int_{\p C_\delta \cap B_1} U_n^2 \varphi_1 \nabla g_1 \cdot \nu d\sigma $$
$$+\lim_
{\eps \to 0} \int_{\p P_\eps \cap (B_1 \setminus C_\delta)} U_n^2 \varphi_1 \nabla g_1 \cdot \nu d\sigma.$$
It is easily seen that the last limit goes to zero, hence our claim is proved. In fact, in the region of integration we have
$$U_n \leq C\eps, $$
and
$$|\nabla (U_n g_1)|, |\nabla U_n| \leq C,$$
from which it follows that
$$U_n |\nabla g_1| \leq C.$$

Finally using the formula for $U$ we can compute that
\begin{equation}\label{delta1}\lim_{\delta \to 0} \int_{\p C_\delta \cap B_1} U_n^2 \varphi_1 \nabla g_1 \cdot \nu d\sigma = \pi \int_L b_1(x') \varphi_1(x', 0, 0) dx',\end{equation}
\begin{equation}\label{delta2}\lim_{\delta \to 0} \int_{\p C_\delta \cap B_1} \bar U_n^2 \varphi_2 \nabla g_2 \cdot \nu d\sigma = \pi \int_L b_2(x') \varphi_1(x', 0, 0) dx'.\end{equation}
Precisely, say for $g_1,$
\begin{align*}
\int_{\p C_\delta \cap B_1} U_n^2 \varphi_1 \nabla g_1 \cdot \nu d\sigma &=\frac 1 \delta  \int_{\p C_\delta \cap B_1} \cos^2 (\frac \theta 2) (g_1)_r \varphi_1 d\sigma\\ &= \int_{\p C_1 \cap B_1}  \cos^2 (\frac \theta 2)((g_1)_r \varphi_1)(X', \delta \cos \theta, \delta \sin \theta) dx' d \theta,\end{align*}
and the desired equality follows.

Combining \eqref{g1}-\eqref{g2}-\eqref{delta1}-\eqref{delta2} with the fact that $\varphi_2=\varphi_2$ on $L,$ and the computation at the end of Lemma \ref{Unhisharm}, we conclude the proof.
\end{proof}

The key ingredient in showing the regularity of the minimizing pair is the following Harnack inequality.

\begin{lem}[Harnack inequality] \label{HImin} Let $(g_1,g_2)$ be a minimizer to $J$ in $B_1$ among competitors which are even-in-$z$, and assume $|g_i| \leq 1$. Then $g_i \in C^\alpha(B_{1/2})$ and $$[g_i]_{C^{0,\alpha}(B_{1/2})} \leq C,$$ with $C$ universal.
\end{lem}

Indeed, since our linear problem is invariant under translations in the $x'$-direction, we see that discrete differences  of the form 
$$g_i(X + \tau) - g_i(X),$$with $\tau$ in the $x'$-direction are also minimizers. Thus by standard arguments (see \cite{CC}) we obtain the following corollary. 

\begin{cor}\label{derivativeH} Let $(g_1,g_2)$ be as in Lemma $\ref{HImin}$. Then $D_{x'}^\beta g_i \in C^{0,\alpha}(B_{1/2})$ and $$[D_{x'}^\beta g_i]_{C^{\alpha}(B_{1/2})} \leq C,$$ with $C$ depending on $\beta.$
\end{cor}

The proof of Harnack inequality relies on the following comparison principle for minimizers.

\begin{lem}[Comparison Principle]\label{mincomp} Let $(g_1,g_2)$ and $(h_1, h_2)$ be minimizers to $J$ in $B_1$. If $$g_i\geq h_i \quad \text{a.e in $B_1 \setminus B_{\rho}, \quad i=1,2$,}$$ then  $$g_i\geq h_i \quad \text{a.e. in $B_1, \quad i=1,2.$}$$  
\end{lem}
\begin{proof}
Call $$M_i: =\max\{g_i, h_i\}, \quad m_i:=\min\{g_i,h_i\}, \quad i=1,2.$$
We have that 
\begin{equation}\label{same_energy} J(M_1, M_2)+J(m_1, m_2) = J(g_1,g_2)+J(h_1,h_2).
\end{equation} On the other hand, since $(g_1,g_2)$ and $(h_1,h_2)$ are minimizing pairs, and under our assumptions $(M_1,M_2), (m_1,m_2)$ are admissible competitors,
$$J(M_1, M_2) \geq J(g_1,g_2), \quad J (m_1,m_2) \geq J(h_1,h_2).$$
We conclude that 
$$J(M_1, M_2) = J(g_1,g_2), \quad J (m_1,m_2) = J(h_1,h_2),$$
and by uniqueness, 
$$M_i=g_i, \quad m_i=h_i,\quad i=1,2.$$
Our desired result follows.
\end{proof}

\textit{Proof of Lemma $\ref{HImin}$.}
The key step consists in proving the following  claim. The remaining ingredients are the standard  Harnack inequality and Boundary Harnack inequality for harmonic functions.

\

{\it Claim:} There exist universal constants $\eta, c$ such that if $g_i \geq 0$ a.e. in $B_1$ and
$$g_1( \frac 1 4 e_n) \geq  1,$$ then
$$g_i \geq  c \quad \text{a.e. in $B_\eta.$}$$
By boundary Harnack inequality, since $U_n g_1$ and $U_n$ are non-negative and harmonic in $B_1 \setminus P^-$ and vanish on $P^-$ we conclude that
$$g_1 \geq c \quad \text{a.e. on $C:=\{|x'| \leq 3/4, 1/8 \leq r \leq 1/2\}$}$$.
Let 
$$v_1(X) := -\frac{|x'|^2}{n-1} + 2(x_n+1) r, \quad v_2(X) := -\frac{|x'|^2}{n-1} - 2(x_n+1) r.$$

We show that
$(v_1,v_2)$ is a  minimizer to $J$ in $B_1$. To do so we prove that $(v_1,v_2)$ satisfies Lemma \ref{char}. 
%$$\int_{B_1} U_n^2 \nabla v \cdot \nabla \phi \;dX = 0, \quad  \forall \phi \in C_0^\infty(B_1). $$

To prove that  \begin{equation*}\Delta (U_n v_1) = 0 \quad \text{in $B_1 \setminus P^-,$}\end{equation*} we use that $2r U_n = U$ and that $U, U_n$ are harmonic outside of $P^-$ and do not depend on $x'$. Thus

$$\Delta(U_n v_1) = -\Delta(\frac{|x'|^2}{n-1}U_n) + \Delta((x_n+1) U) = -2U_n + 2U_n=0.$$ We argue similarly for $v_2.$

Finally the fact  that
$$\lim_{r \rightarrow 0} (v_1)_r + (v_2)_r = 0,$$ follows immediately from the definition of $v_i, i=1,2$.

Now by choosing $\delta$ small depending on $c$ we can guarantee that 
$$v_1 + \delta \leq c \leq g_1, \quad v_2+\delta \leq 0 \leq g_2 \quad \text{a.e. on $\{|x'| \leq 3/4, \delta/2 \leq r \leq \delta\},$}$$
$$v_i + \delta \leq 0 \leq g_i, \quad i=1,2 \quad \text{on $\{|x'| = 3/4, r \leq \delta\}.$}$$

By our comparison principle, Lemma \ref{mincomp} we conclude that
$$v_i + \delta \leq g_i \quad \text{a.e in $\{|x'|\leq 3/4, r \leq \delta.\}$}$$ Since $v_i+\delta$ is a continuous function positive at zero, our claim follows immediately.
\qed

\

\textit{Proof of Theorem $\ref{class}$.} Let $(g_1,g_2)$ be a minimizer to $J$ in which achieves the boundary data $(h_1,h_2).$ By approximation, we can assume that $h_i$ is smooth. Then by Lemma \ref{Unhisharm}, $(g_1,g_2)$ satisfy  $$\Delta (U_n g_1) = 0 \quad \text{in $B_1 \setminus P^-,$}$$
$$\Delta (-\bar U_n g_2) = 0 \quad \text{in $B_1 \setminus P^+,$}$$
and in view of Lemma \ref{HImin} and Corollary \ref{derivativeH}, $\Delta_{x'} g_i \in C^{0,\alpha}(B_1)$ with universally bounded norm in say $B_{3/4}$. Thus,  it is shown in Theorem 8.1 in \cite{DR} that (after reflecting $g_2$)
\begin{equation}\label{e2}|g_i(x', x_n, z)-g_i(x', 0,0) -b_i(x') r| = O(r^{3/2}), \quad (x',0,0) \in B_{1/2} \cap L,\end{equation} with $b_i$ Lipschitz and moreover,
\begin{equation}\label{hb}|(g_i)_r(x',x_n,z) -b_i(x') | \leq Cr^{1/2}, \quad (x',0,0) \in B_{1/2} \cap L,\end{equation}with $C$ universal. 
Hence by Lemma \ref{char},
$$(g_1)_r(x') + (g_2)_r(x')= b_1(x')+b_2(x')=0 \quad \text{on $B_{1/2}\cap L$}.$$
For completeness, we sketch the argument to obtain \eqref{e2}-\eqref{hb}, say for $g_1$.

Since $g_1$ solves $$\Delta(U_n g_1) =0 \quad \text{in $B_1 \setminus P^-$}$$ and $U_n$ is independent on $x'$ we can rewrite this equation as
\begin{equation}\label{2dreduction} \Delta_{x_n,z} (U_n  g_1) = - U_n \Delta_{x'} g_1 ,\end{equation}
and according to Corollary \ref{derivativeH} we have that 
$$\Delta_{x'} g_1 \in C^{\alpha}(B_{1/2}),$$ with universal bound. 
Thus, for each fixed $x'$, we need to investigate the 2-dimensional problem 
$$\Delta (U_t g_1) = U_t f, \quad \text{in $B_{1/2} \setminus \{t \leq 0, z=0\}$}$$
with $$f \in C^{\alpha}(B_{1/2}),$$ and $g_1$ bounded. Without loss of generality, for a fixed $x'$ we may assume $g_1(x', 0,0)=0.$

Let $H(t,z)$ be the solution to the problem
$$\Delta H = U_t f, \quad \text{in $B_{1/2} \setminus \{t \leq 0, z=0\},$}$$
such that 
$$H = U_t  h  \quad \text{on $\p B_{1/2}$}, \quad H=0 \quad \text{on $B_{1/2} \cap \{t \leq 0, z=0\}.$}$$

We wish to prove that \begin{equation}\label{equal}U_t h = H.\end{equation}

First notice that  $$ \Delta (H - U_t  h) =0 \quad  \text{in $B_{1/2} \setminus \{t \leq 0, z=0\},$}$$ and $$H = 0 \quad \text{on $\p B_{1/2} \cup (B_{1/2} \cap \{t < 0, z=0\})$}.$$
We claim that
\begin{equation}\label{limitzero}
\lim_{(t,z) \rightarrow (0,0)} \frac{H- U_t h}{U_t}=\lim_{(t,z) \rightarrow (0,0)} \frac{H}{U_t}=0.
\end{equation}
If the claims holds, then given any $\eps >0$
$$-\eps U_t \leq H - U_t h \leq \eps U_t, \quad \text{in $B_\delta$}$$
with $\delta=\delta(\eps).$
Then by the maximum principle the inequality above holds in the whole $B_{1/2}$ and by letting $\eps \rightarrow 0$ we obtain \eqref{equal}.

To prove the claim \eqref{limitzero} we show that $H$ satisfies the following
\begin{equation}\label{Hexp}
|H(t,z) - aU(t,z)| \leq C_0 r^{1/2}U(t,z), \quad r^2=t^2+z^2, a\in \R,
\end{equation} with $C_0$ universal.

To do so, we consider the holomorphic transformation 
$$\Phi: (s,y) \rightarrow (t,z) =(\frac 1 2 (s^2-y^2), sy)$$
which maps  $B_{1} \cap \{s > 0\}$ into $B_{1/2} \setminus  \{t \leq 0, z=0\}$ and call

$$\tilde H (s,y) = H(t,z), \quad \tilde f(s,y) = f(t,z).$$

Then, easy computations show that
$$\Delta \tilde H = s \tilde f \quad \text{in  $B_{1} \cap \{s > 0\}$}, \quad \tilde H=0 \quad \text{on $B_{1} \cap \{s=0\}$.}$$
Since the right-hand side is $C^\alpha$ we conclude that $\tilde H \in C^{2,\alpha}$. In particular $\tilde H_s$ satisfies 
$$|\tilde H_s (s, y) -a| \leq C_0|(s,y)|, \quad a=\tilde H_s(0,0)$$ with $C_0$ universal.
Integrating this inequality between $0$ and $s$ and using that $\tilde H= 0$ on $B_{1} \cap \{s=0\}$ we obtain that
$$|\tilde H(s,y) - a s| \leq C_0 s|(s,y)|. $$
In terms of $H$, this equation gives us
\begin{equation}\label{HaU}|H - a U| \leq C_0 r^{1/2} U\end{equation} as desired.

Thus \eqref{equal} and \eqref{Hexp} hold and by combining them and using that $U/U_t= 2r$ we deduce
$$|h -  2 a r| \leq 2 C_0 r^{3/2},$$
which is  the desired estimate i.e. (recall that above we assumed $g_1(x',0,0)=0$)
\begin{equation*}\label{exp3} |g_1(x',x_n,x) - g_1(x', 0,0) - b_1(x') r| \leq 2C_0 r^{3/2}, \end{equation*}
with $$b_1(x') = 2\tilde H_s (x',0,0).$$

We remark that $b(x')$ is Lipschitz. Indeed, notice that the derivatives $(g_1)_i, i=1,\ldots, n-1$ still satisfy the same equation \eqref{2dreduction} as $g_1$, where the $C^\alpha$ norm of the right-hand side has a universal bound. Thus, we can argue as above to conclude that 
%$\p_i \tilde H(x', s, y) \in C^1$ with universal bound and hence 
$$|\p_i \tilde H_s(x',0,0)| \leq C,$$ which together with the formula for $b_1(x')$ shows that $b_1(x')$ is a Lipschitz function.

To obtain the second of our estimates \eqref{hb} we proceed similarly as above. 
Since $U_t h= H$ one can compute easily that \begin{equation}\label{hr}U_t h_r = H_r + \frac{1}{2}\frac{H}{r}.\end{equation}
Moreover, after our holomorphic transformation
\begin{equation}\label{holoradial}2r H_r(t,z) = s \tilde H_s(s,y)+ y \tilde H_y(s,y).\end{equation}
As observed above,
$$|\tilde H_s -a(x')|\leq C |(s,y)|,$$
and similarly since  $\tilde H= 0$ on $B_{1} \cap \{s=0\}$ 
$$|\tilde H_y| \leq C s.$$
These two inequalities combined with \eqref{holoradial} give us 
\begin{equation}\label{2rHr}|2r H_r -  a(x') U| \leq Cr^{1/2}U.\end{equation}
Combining \eqref{hr} with \eqref{HaU}-\eqref{2rHr} we obtain \eqref{hb} as desired.

Finally, since $g_i$ is $C^\infty$ in the $x'$-direction and $g_1=g_2$ on $L$, we conclude that for a fixed $X_0 \in B_{1/4} \cap L$
$$|g_i(x',0,0)- g_i(x'_0,0,0) - a'\cdot (x'-x'_0)| \leq C|x'-x'_0|^2$$
from which \eqref{mainh} follows, using that $b_i$ is Lipschitz continuous.

We are left with the proof that the boundary data is achieved continuously.
Indeed this follows by classical elliptic theory if we restrict to $\p B_1 \setminus P^\mp$  (for $g_1$ and $g_2$ respectively.)

If $X_0 \in \p B_1 \cap (P^- \setminus L)$ then in a small neighborhood of $X_0$ intersected with $B_1 \cap \{z >0\}$ the function $U_n g_1$ is harmonic continuous up to the boundary and vanishes continuously on $\{z=0\}$ (since $h$ is bounded). The continuity of $g_1$ at $X_0$ then follows from standard boundary regularity for the harmonic function $U_n g_1$ after reflecting oddly across $z=0$ and using that the boundary data is smooth. We argue similarly if 
$X_0 \in \p B_1 \cap (P^+ \setminus L)$ and we must show the continuity of $g_2$ at $X_0$.

Finally, on the set $\p B_1 \cap L$ as in the case of Laplace equation, it suffices to construct at each point $X_0$ a local barrier pair (minimizing pair) for $(g_1,g_2)$ which is zero at $X_0$ and strictly negative in a neighborhood of $X_0$. 
Such barrier is given by  a multiple of (see Lemma \ref{char}) $$((x'-x_0') \cdot x_0', (x'-x_0') \cdot x_0').$$  
\qed

\section{Harnack Inequality}\label{sec:harnack}

This section is devoted to determine a Harnack type inequality for solutions to our free boundary problem \eqref{FBintro}.

\begin{thm}[Harnack inequality]\label{mainH} There exists $\bar \eps > 0$  such that if $(u_1,u_2)$ solves \eqref{FBintro} and it satisfies in $B_\rho(X^*), $
\begin{equation*}\label{flatH}U(X +\eps a_0 e_n) \leq u_1(X) \leq U(X+\eps b_0e_n), \quad
\bar U(X +\eps \bar a_0 e_n) \leq u_2(X) \leq \bar U(X+\eps \bar b_0e_n), \end{equation*}with
$$\eps (b_0 - a_0) \leq \bar \eps \rho, \quad \eps (\bar a_0 - \bar b_0) \leq \bar \eps \rho,$$  then 
\begin{equation*}U(X +\eps a_1 e_n) \leq u_1(X) \leq U(X+\eps b_1e_n) \quad \textrm{in $B_{\eta \rho}(X^*)$, }\end{equation*} 
\begin{equation*}\bar U(X +\eps \bar a_1 e_n) \leq u_2(X) \leq \bar U(X+\eps \bar b_1e_n) \quad \textrm{in $B_{\eta \rho}(X^*)$, }\end{equation*} with $$a_0 \leq a_1\leq b_1 \leq b_0, \quad (b_1-a_1) \leq (1-\eta)(b_0-a_0),$$ 
$$\bar b_0 \leq \bar b_1\leq \bar a_1 \leq \bar a_0, \quad (\bar a_1-\bar b_1) \leq (1-\eta)(\bar a_0- \bar b_0),$$
for a small universal constant $\eta$. Moreover, if $B_\rho(X^*) \cap F(u_1,u_2) \neq \emptyset,$ and $a_0=\bar b_0, b_0=\bar a_0$ then the conclusion holds with $a_1=\bar b_1, b_1=\bar a_1.$\end{thm}

Let $(u_1,u_2)$ be as in the Theorem above, and denote 
$A^\mp_\eps$ be the following sets
\begin{equation}\label{Aeps-} A^-_{\eps} := \{(X, (\tilde u_1)_\eps(X))  : X \in B_{1-\eps} \setminus P^-\} \subset \R^{n+1} \times [a_0,b_0].\end{equation}
\begin{equation}\label{Aeps+} A^+_{\eps} := \{(X, (\tilde {\bar u}_2)_\eps(X))  : X \in B_{1-\eps} \setminus P^+\} \subset \R^{n+1} \times [\bar b_0,\bar a_0].\end{equation}
Since $\eps$-domain variations may be multivalued, we mean that given $X$ all pairs $(X, Z)$  with $Z$ an element of the $\eps$-variation belong to $A^\mp_\eps$. Applying Theorem \ref{mainH} iteratively we obtain
\begin{equation}\label{oscA-} A^-_\eps \cap (B_{\frac 1 2 \eta^{m} - \eps}(X^*) \times [a_0,b_0]) \subset B_{\frac 1 2 \eta^{m} - \eps}(X^*) \times [a_m,b_m],\end{equation}  
\begin{equation}\label{oscA+} A^+_\eps \cap (B_{\frac 1 2 \eta^{m} - \eps}(X^*) \times [\bar b_0,\bar a_0]) \subset B_{\frac 1 2 \eta^{m} - \eps}(X^*) \times [\bar b_m,\bar a_m],\end{equation} 
with
\begin{equation}\label{osctilde}b_m-a_m \leq (b_0-a_0)(1-\eta)^m, \quad \bar a_m -\bar b_m \leq (\bar a_0-\bar b_0)(1-\eta)^m\end{equation} for all $m$'s such that both inequalities are satisfied
\begin{equation}\label{m}2\eps (1-\eta)^m \eta^{-m}(b_0-a_0) \leq 
\bar \eps, \quad 2\eps (1-\eta)^m \eta^{-m}(\bar a_0-\bar b_0) \leq 
\bar \eps .\end{equation}

Thus we obtain the following corollary, which will be the key of the improvement of flatness arguments of Section \ref{sec:improvement}.

\begin{cor} \label{corHI}If  
$$U(X - \eps e_n) \leq u_1(X) \leq U(X+\eps e_n) \quad \textrm{in $B_1$,}$$ 
$$\bar U(X +\eps e_n) \leq u_2(X) \leq \bar U(X-\eps e_n) \quad \textrm{in $B_1$,}$$ 
with $\eps \leq \bar \eps/2$, given $m_0>0$ such that $$2\eps (1-\eta)^{m_0} \eta^{-m_0} \leq \bar\eps,$$ then the set $A^\mp_\eps \cap (B_{1/2} \times [-1,1])$  is above the graph of a function $y = a^\mp_\eps(X)$ and it is below the graph of a function $y = b^\mp_\eps(X)$ with
$$ b^\mp_\eps - a^\mp_\eps \leq 2(1 - \eta)^{m_0-1},$$
and $a^\mp_\eps, b^\mp_\eps$ having a modulus of continuity bounded by the H\"older function $\alpha t^\beta$ for $\alpha, \beta$  depending only on $\eta$. Moreover, $a^-_\eps=a^+_\eps, b^-_\eps=b^-_\eps$ on $B_{1/2}\cap F(u_1,u_2).$
\end{cor}

The following lemma is the key ingredient in the proof of Theorem \ref{mainH}.

\begin{lem}\label{babyH}
There exists $\bar \eps > 0$ such that for all  $0 < \eps \leq \bar \eps$ if $(u_1,u_2)$ is a solution to \eqref{FBintro}  in $B_1$ such that  
\begin{equation}u_1(X) \geq U(X), \quad u_2(X) \leq \bar U(X) \quad \text{in $B_{1/2},$}\end{equation} and at $\bar X  \in B_{1/8}( \frac{1}{4} e_n)$
\begin{equation}\label{Bound} u_1(\bar X) \geq U(\bar X+\eps e_n), 
\end{equation} then
 \begin{equation}\label{onesideimprov} 
 u_1(X) \geq U(X + \tau \eps e_n), \quad u_2 \leq \bar U(X+\tau \eps e_n)  \quad \textrm{in $B_{\delta}$},  \end{equation} for universal constants $\tau, \delta.$
\end{lem}

Lemma \ref{babyH}  relies on building an appropriate family of radial subsolutions. Precisely, 
let $R>0$ and denote by $$V_R(t,z) = U(t,z)((n-1)\frac{t}{R} + 1 ), \quad \bar V_R(t,z) = \bar U(t,z)((n-1)\frac{t}{R} + 1 ). $$ Then set, 
\begin{eqnarray*}\label{v1v2R}
v_R(X)= V_R(R- \sqrt{|x'|^2+(x_n-R)^2}, z)\\
\ \\
\bar v_R(X)= \bar V_R(R- \sqrt{|x'|^2+(x_n-R)^2}, z).\end{eqnarray*}
Finally, for $\beta>0$, define (for notational simplicity we drop the dependence on $\beta$,)
\begin{equation}
v_1^R=(1+\frac{\beta}{R})v_R, \quad v_2^R=(1+\frac{\beta}{R}){\bar v}_R.
\end{equation}

\begin{prop}\label{sub} If $R$ is large enough, the ordered pair $(v_1^R, v_2^R)$ is a comparison subsolution to \eqref{FBintro} in $B_2$ which is strictly monotone in the $e_n$-direction outside $\{v_i^R(x,0)=0\}$. Moreover, there exist functions $\tilde v_1^R, \tilde{\bar v}_2^R$ such that 
\begin{equation}\label{eq}
U(X) = v_1^R(X - \tilde v_1^R(X)e_n), \quad \text{in $B_1\setminus P^-$}, \quad \bar U(X) = v_2^R(X - \tilde {\bar v}_2^R(X)e_n), \quad \text{in $B_1\setminus P^+$}
\end{equation}
and
\begin{equation}\label{estvr}
|\tilde v_1^R(X) - \gamma_R(X)| \leq \frac{C}{R^2} , \quad |\tilde {\bar v}_2^R(X) - \bar \gamma_R(X)| \leq \frac{C}{R^2} ,
\end{equation} with
\begin{equation}\label{estvr2}
\gamma_R(X)=- \frac{|x'|^2}{2R} + 2(n-1)\frac{x_n r}{R} + 2 \beta \frac{r}{R}, \quad \bar \gamma_R(X)=- \frac{|x'|^2}{2R} - 2(n-1)\frac{x_n r}{R}- 2 \beta \frac{r}{R}
\end{equation} 
for $r= \sqrt{x_n^2+z^2}$ and $C$ depending on $\beta$.\end{prop}
\begin{proof}  
{\bf Step 1.} We need to show that $v_1^R$ is strictly subharmonic in $B_2^+(v_1^R)$ while $v_2^R$ is strictly superharmonic  in $B_2^+(v_2^R)$. We sketch the proof of the claim about $v_2^R.$ The claim about $v_1^R$ can be obtained by similar computations which are already contained in Proposition 6.4 in \cite{DR} to which we refer for further details.

%From the formula for $v_R$ we see immediately that ($R>2$)
%$$B_2^+(v_R) = B_2 \setminus (\mathcal{B}_2  \setminus \overline{\mathcal{B}}_R(Re_n)).$$
%and $$F(v_R) = \p \mathcal{B}_R(0,R) \cap \mathcal{B}_2.$$ 
On the set $B_2^+(v_2^R)$, 

$$\Delta \bar v_R (X) = (( \p_{tt}+\p_{zz}))\bar V_R)(R-\rho, z) - \frac{n-1}{\rho}\p_t \bar V_R(R-\rho, z), $$
where $$\rho :=  \sqrt{|x'|^2+(x_n-R)^2}.$$ 
Also for $(t,z)$ outside the set  $\{(t,0) : t \geq 0\} $
\begin{align*}\Delta_{t,z} \bar V_R (t,z)&= (\p_{tt}+ \p_{zz})\bar V_R(t,z) = \frac{2(n-1)}{R} \p_t \bar U(t,z)+ (1+(n-1)\frac{t}{R})\Delta_{t,z} \bar U(t,z)\\
&= \frac{2(n-1)}{R} \p_t \bar U(t,z),\end{align*}
and 
\begin{equation}\label{pV}\p_t \bar V_R(t,z) = (1+(n-1)\frac{t}{R})\p_t \bar U(t,z) + \frac{n-1}{R} \bar U(t,z).\end{equation}
Thus we need to prove that in $B^+_2(v_2^R)$ 
$$\frac{2(n-1)}{R} \p_t \bar U - \frac{n-1}{\rho}[(1+(n-1)\frac{R-\rho}{R})\p_t \bar U + \frac{n-1}{R} \bar U] < 0,$$
where $\bar U$ and $\p_t \bar U$ are evaluated at $(R-\rho,z).$ 

Set $ t = R -\rho$, then straightforward computations reduce the inequality above to
$$[2(R- t) - R - (n-1) t]\p_t \bar U(t, z) - (n-1)\bar U(t, z) < 0.$$
Using that $\p_t \bar U( t, z)=- \bar U( t,z)/(2 r)$ with $r^2= t^2 + z^2$, this inequality is immediately satisfied 
for $R$ large.

Now we prove that $(v_1^R, v_2^R)$ satisfies the free boundary condition.
Observe that $$F(v_1^R,v_2^R) = \p \mathcal{B}_R(Re_n,0) \cap \mathcal{B}_2,$$ 
and hence it is smooth. By the radial symmetry it is enough to show that the free boundary condition is satisfied at $0$. Precisely we have,
\begin{equation}\label{expa1}v_R(x,z) = U(x_n,z) + o(|(x,z)|^{1/2}), \quad \text{as $(x,z) \rightarrow (0,0),$}\end{equation}
and
\begin{equation}\label{expa2}\bar v_R(x,z) = \bar U(x_n,z) + o(|(x,z)|^{1/2}), \quad \text{as $(x,z) \rightarrow (0,0).$}\end{equation}

Indeed, let us prove the expansion for $\bar v_R.$ Since $\bar U$ is H\"older continuous with exponent $1/2$, it follows that
$$|\bar V_R(t,z) - \bar V_R(t_0,z)| \leq C |t-t_0|^{1/2} \quad \text{for $|t-t_0| \leq 1.$}$$
Thus for $(x,z) \in B_s,$ $s$ small
$$|\bar v_R(x,z) - \bar V_R(x_n,z)| = |\bar V_R(R-\rho, z) - \bar V_R(x_n, z)| \leq C |R-\rho - x_n|^{1/2} \leq C s.$$
%where we have used that (recall that $\rho :=  \sqrt{|x'|^2+(x_n-R)^2}$) 
%\begin{equation}\label{easy}R - \rho - x_n = - \frac{|x'|^2}{R-x_n + \rho}.\end{equation}
Hence for $(x,z) \in B_s$
\begin{align*}|\bar v_R(x,z) - \bar U(x_n,z)| & \leq Cs + |\bar V_R(x_n,z) - \bar U(x_n, z)|.\end{align*} Thus from the formula for $\bar V_R$
$$|\bar v_R(x,z) - \bar U(x_n,z)| \leq Cs +(n-1) \frac{|x_n|}{R}\bar U(x_n,z) \leq C' s, \quad (x,z) \in B_s$$
which gives the desired expansion \eqref{expa2}. 

Now, we show that $v_2^R$ is strictly monotone decreasing in the $e_n$-direction in $B_2^+(v_2^R)$. Outside of its zero plate,
$$\p _{x_n} \bar v_R(x) = - \frac{x_n-R}{\rho} \p_t \bar V_R(R-\rho, z).$$ Thus we only need to show that $\bar V_R(t,z)$ is strictly monotone decreasing in $t$ outside $\{(t,0) : t \geq 0\}$ . This follows immediately from \eqref{pV} and the formula for $\bar U$.

\

{\bf Step 2.} We show the existence of $\tilde {v}_1^R$ satisfying \eqref{eq} and \eqref{estvr}. The claim for $v_2^R$ can be proved with similar arguments.

Following Proposition 3.5 in \cite{DS} it suffices to show the 2-dimensional claim:
$$U(t+ \gamma_R(t,z) - \frac{C}{R^2}, z) \leq v_{\beta, R}(t,z) \leq U(t+ \gamma_R(t,z) + \frac{C}{R^2}, z) $$
with
$$v_{\beta, R}(t,z):=(1+\frac \beta R) V_R(t,z),$$
and
$$\gamma_R(t,z)= 2\frac{(n-1)}{R} rt + 2\beta \frac{r}{R}, \quad r=t^2+z^2.$$

Indeed since (see the properties listed at the beginning of the Appendix,) $$|U_{tt}| \le C r^{-1} U_t$$we have that if $|\mu| \le r/2$ then
$$|U(t+\mu,z)-(U(t,z)+\mu U_t(t,z))| \le \mu^2 |U_{tt}(t',z)| \le C \mu^2 r^{-1} U_t(t,z) ,$$
where in the last inequality we used \eqref{diadic}. Thus, since $U_t=U/(2 r)$,
$$(1+\frac {\mu}{2 r}+ C \frac{\mu^2}{r^2})U(t,s) \ge U(t+\mu,z) \ge (1+\frac {\mu}{2 r} - C \frac{\mu^2}{r^2})U(t,z) .$$
Choosing $$\mu = \tilde \mu \pm 4C \frac{\tilde \mu^2}{r}$$ we obtain that 
\begin{equation}\label{taylor1}U(t+\tilde \mu + 4C \frac {\tilde \mu^2}{r},z) \ge (1+\frac{\tilde \mu}{ 2 r})U(t,z) \ge U(t+\tilde \mu - 4C \frac {\tilde \mu^2}{r},z), \end{equation} provided that $|\tilde \mu| / r \leq c$, with $c$ sufficiently small.
Since $$v_{\beta, R}=(1+\frac{n-1}{R}t + \frac{\beta}{R}+ \frac{(n-1)\beta}{R^2} t) \,U$$ we can apply the inequality above with
$$\tilde \mu = 2r (\frac{n-1}{R}t + \frac{\beta}{R}+ \frac{(n-1)\beta}{R^2} t)$$ hence $|\tilde \mu|/r \le \frac{C}{R}$ and obtain the claim.

\end{proof}

Furthermore, these radial subsolutions satisfy the following estimates.

\begin{cor}\label{corest}There exist $\delta, c_0,C_0, C_1$ depending on $\beta$, such that 
\begin{align}
\label{4}& v_1^R(X + \frac{c_0}{R}e_n) \geq U(X + \frac{c_0}{2R} e_n), \quad  v_2^R(X + \frac{c_0}{R}e_n) \leq \bar U(X + \frac{c_0}{2R} e_n), \quad \text{in $B_{\delta},$}\\
\label{3}& v_1^R(X - \frac{C_1}{R} e_n) \leq U(X), \quad v_2^R(X - \frac{C_1}{R} e_n) \geq \bar U(X), \quad \text{in $\overline{B}_1,$}\end{align}
\begin{equation}
%\label{1}& v_R(X+ \frac{c_0}{R}e_n) \leq U(X) \quad \text{in $\{X\in B_1 : |x' | \geq 1/2, |x_n|  \leq \delta_0\},$}\\
\label{2} v_1^R(X+ \frac{c_0}{R}e_n) \leq (1+\frac{C_0}{R}) U(X), \quad \text{in $\overline{B}_1 \setminus B_{1/4}$},\end{equation} with strict inequality on $F(v_1^R(X+ \frac{c_0}{R} e_n)) \cap ( \overline{B}_1 \setminus B_{1/4}),$
and
\begin{equation}
%\label{1}& v_R(X+ \frac{c_0}{R}e_n) \leq U(X) \quad \text{in $\{X\in B_1 : |x' | \geq 1/2, |x_n|  \leq \delta_0\},$}\\
\label{5} v_2^R(X+ \frac{c_0}{R}e_n) \geq (1-\frac{C_0}{R})\bar U(X), \quad \text{in $\overline{B}_1 \setminus B_{1/4}$},
\end{equation} with strict inequality on $L \cap (\overline{B}_1 \setminus B_{1/4}).$
\end{cor}

\begin{proof} Estimates \eqref{4}-\eqref{3} follow from \eqref{eq} and Lemmas \ref{elem}-\ref{elem2}. Let us prove \eqref{5}. The proof of \eqref{2} can be obtained similarly.

Using again \eqref{eq} and Lemma \ref{elem2} we get that
$$v_2^R(X+ \frac{c_0}{R}e_n) \geq \bar U(X) \quad \text{in $B_1 \cap \{|x'| \geq \frac 1 8, |x_n| \leq \bar \delta\},$}$$
with $c_0, \bar \delta$ depending on $\beta$ and $R$ large (with strict inequality on $L$.)
Thus we need to show that \eqref{5} holds on the complement of $(\bar B_1 \setminus B_{1/4}) \cap \{r \leq \bar \delta\}.$ On the other hand in $B_1,$
$$v_2^R(X + \frac{c_0}{R}e_n) \geq \bar U(X + \frac{\bar C}{R}e_n)$$
for some $\bar C$ depending on $\beta$ (combining again \eqref{eq} and Lemma \ref{elem2}.) From the version of Lemma \ref{cor2} in the Appendix for the function $\bar U$, we conclude that
$$\bar U(X + \frac{\bar C}{R}e_n) \geq (1-C \frac{\bar C}{R}) \bar U(X),$$
as long as $r > \bar \delta$ and with $C=C(\bar \delta).$ The desired conclusion immediately follows.
\end{proof}

Having established  the Proposition and the Corollary above, we are ready to present the proof of Lemma \ref{babyH}.

\

\textit{Proof of Lemma $\ref{babyH}$.} In view of \eqref{Bound}, we have
$$u_1(\bar X) - U(\bar X) \geq U(\bar X+\eps e_n) - U(\bar X)= \p_t U(\bar X + \lambda e_n)\eps \geq c \eps, \quad \lambda \in (0,\eps).$$
Hence by Lemma \ref{basic} in the Appendix it follows that
\begin{equation}\label{51} u_1(X) \geq (1+c'\eps)U(X) \quad \text{in $\bar B_{1/4}.$}
\end{equation}
Now set,
$$w_i^{R,t}(X):=(1+c''\eps) v_i^R(X+te_n), \quad X \in B_1$$
with $(v_1^R, v_2^R)$ the subsolution pair defined in \eqref{v1v2R} for $\beta=0$, $c'' \leq c'$ and $R$ large to be chosen later. 

%By Lemma \ref{basic} it also follows that 
%$$(1+ \frac{c'}{4} \eps) U(X) \leq U(X+C'\eps e_n), \quad \text{in $B_1$},$$ with $C'$ universal.

If $t_0= -C_1/R$, then it follows from \eqref{3} and the inequality above that
$$w_1^{R, t_0} \leq (1+c'\eps )U(X) \leq u_1, \quad w_2^{R, t_0} \geq \bar U \geq u_2 \quad \text{in $\bar B_{1/4}$}.$$
If $t_1 = c_0/R$, then choosing $R=2C_0/c'' \eps$, and $\eps$ small (depending on $c''$), \eqref{2} gives
$$w_1^{R,t_1} \leq (1+c'\eps)U \leq u_1 \quad \text{on $\p B_{1/4}$}$$
while \eqref{5} gives
$$w_2^{R,t_1} \geq (1+c'' \eps)(1-\frac{C_0}{R})\bar U \geq \bar U \geq u_2 \quad \text{on $\p B_{1/4}$},$$
with the inequality being strict on the necessary sets to apply Corollary \ref{compmon}. Thus, we conclude that 
$$w_1^{R, t_1} \leq u_1, \quad w^2_{R,t_2} \geq u_2 \quad \text{in $B_{1/4}.$}$$ From \eqref{4} it follows that in $B_\rho,$ for $\rho \leq \delta,$
$$u_1 \geq U(X + \frac{c_0}{2R}e_n)$$ while
$$u_2 \leq (1+ c''\eps)\bar U(X + \frac{c_0}{2R}e_n)\leq \bar U(X+ \frac{c_0}{2R}e_n - Cc'' \rho \eps e_n).$$
The last inequality follows from a variant of Lemma \ref{basic} for $\bar U$, rescaled in the ball $B_\rho.$ Therefore, for $\rho$ small (depending on $c'', C$,) we get that 
$$u_2 \leq \bar U(X + \frac{c_0}{4R}) \quad \text {in $B_\rho$},$$
and the desired gain is achieved.
\qed

\section{Improvement of Flatness}\label{sec:improvement}

In this section we show our main improvement of flatness theorem, from which the desired regularity Theorem \ref{mainT} follows with standard arguments (see for example \cite{CC}.)

\begin{thm}[Improvement of flatness]\label{iflat}There exist $\bar \eps > 0$ and $\rho>0$ universal constants such that for all  $0 < \eps \leq \bar \eps$ if $(u_1,u_2)$ solves \eqref{FBintro}  with $0 \in F(u_1,u_2)$ and it satisfies 
\begin{equation}\label{flatimp1}U(X - \eps e_n) \leq u_1(X) \leq U(X+\eps e_n) \quad \textrm{in $B_1$,}\end{equation}
\begin{equation}\label{flatimp2}\bar U(X + \eps e_n) \leq u_2(X) \leq \bar U(X-\eps e_n) \quad \textrm{in $B_1$,}\end{equation}
 then 
\begin{equation}\label{flatimp3}\alpha U(x \cdot \nu  - \frac \eps 2 \rho, z) \leq u_1(X) \leq \alpha U(x\cdot \nu+\frac \eps 2 \rho, z) \quad \textrm{in $B_\rho$},\end{equation}
\begin{equation}\label{flatimp4}\alpha \bar U(x \cdot \nu  + \frac \eps 2 \rho, z) \leq u_2(X) \leq \alpha \bar U(x\cdot \nu-\frac \eps 2 \rho, z) \quad \textrm{in $B_\rho$},\end{equation} 
 for some direction $\nu \in \R^n, |\nu|=1,$ and $|\alpha-1| \leq \eps.$ 
\end{thm}

The proof of Theorem \ref{iflat} relies on the next two lemmas.
The first Lemma is an immediate consequence of Corollary \ref{corHI}.

\begin{lem}\label{ginfty} Let $\eps_k \rightarrow 0$ and let $(u_1^k, u_2^k)$ be a sequence of solutions to \eqref{FBintro} with $0 \in F(u_1^k, u_2^k)$ satisfying \begin{equation}\label{flatimp_k}U(X - \eps_k e_n) \leq u_1^k(X) \leq U(X+\eps_k e_n) \quad \textrm{in $B_1$,}\end{equation}
\begin{equation}\label{flatimp_k2}\bar U(X - \eps_k e_n) \leq u_2^k(X) \leq \bar U(X+\eps_k e_n) \quad \textrm{in $B_1$.}\end{equation}
Denote by  $\tilde u_1^k, \tilde{\bar u}_2^k$ the  $\eps_k$-domain variation of $(u_1^k, u_2^k)$ with respect to $U, \bar U$ respectively.
Then the sequences of sets 
$$A_1^k := \{(X, \tilde u_1^k (X)) : X \in B_{1-\eps_k} \setminus P^-\},$$ 
$$A_2^k := \{(X, \tilde {\bar u}_2^k (X)) : X \in B_{1-\eps_k} \setminus P^+\},$$ 
have a subsequence that converge uniformly (in Hausdorff distance) in $B_{1/2} \setminus P^\mp$ to the graph $$A^\mp_\infty := \{(X,\tilde u^\mp_\infty(X)) : X \in B_{1/2} \setminus P^\mp\},$$ where $\tilde u^\mp_\infty$ is  a H\"older continuous function in $B_{1/2}$ and
$$\tilde u^-_\infty=\tilde u^+_\infty \quad \text{on $B_{1/2}\cap L$.}$$
\end{lem}

\begin{lem}\label{tildesolves} The pair $(\tilde u^-_\infty, \tilde u^+_\infty)$ in Lemma $\ref{ginfty}$ solves the linearized problem \eqref{linear} in $B_{1/2}$. In particular,  in $B_{2\rho} \setminus P^\mp$ ($\rho$ small universal )
 \begin{equation}  a' \cdot x' + b_\mp r-\frac{1}{8} \rho \leq \tilde u^\mp_\infty (X)\leq a' \cdot x' + b_\mp r+ \frac{1}{8}\rho,\end{equation}for some $a'\in \R^{n-1}, b_\mp \in \R$ (universally bounded) such that
$$b_-+b_+=0.$$ 
\end{lem}
\begin{proof}
We start by showing that $U_n \tilde u^-_\infty$ is harmonic in $B_{1/2} \setminus P^-. $ Similarly, we can show that  $-\bar U_n \tilde u^+_\infty$ is harmonic in $B_{1/2} \setminus P^+. $ 

Let $\tilde \varphi$ be a smooth function which touches $\tilde u^-_\infty$ strictly by below at $X_0 \in B_{1/2} \setminus P^-.$ We need to show that 
\begin{equation}\label{des} \Delta (U_n \tilde \varphi)(X_0) \leq 0.
\end{equation}
Since by the previous lemma, the sequence $A^k_1$ converges uniformly to $A^-_\infty$ in $B_{1/2} \setminus P^-$ we conclude that there exist a sequence of constants $c_k \rightarrow 0$ and a sequence of points $X_k \in B_{1/2} \setminus P^-$, $X_k \rightarrow X_0$ such that $\tilde \varphi_k := \tilde \varphi + c_k$ touches $\tilde u^k_1$ by below at $X_k $ for all $k$ large enough. 

Define the function $\varphi_k$ by the following identity
\begin{equation}\label{varphi}\varphi_k (X- \eps_k \tilde{\varphi_k}(X) e_n) = U(X). \end{equation}

Then according to \eqref{gtildeg} $\varphi_k$ touches $u^k_1$ by below at $Y_k = X_k - \eps_k \tilde \varphi_k(X_k)e_n \in B_1^+(u_1^k),$ for $k$ large enough. Thus, since $u_1^k$ satisfies \eqref{FBintro} in $B_1$ it follows that
\begin{equation}\label{sign}
\Delta \varphi_k(Y_k) \leq 0.
\end{equation}
By Proposition \ref{1},
$$\Delta \varphi_k (Y_k)= \eps_k \Delta(U_n \tilde{\varphi})(Y_k) + O(\eps_k^2).$$
Thus, as $k \to \infty$
$$\Delta (U_n \tilde \varphi)(X_0) \leq 0,$$ as desired.

Next we need to show that $$(\tilde u^-_\infty)_r(X_0) + (\tilde u^+_\infty)_r(X_0)=0, \quad X_0=(x'_0,0,0) \in B_{1/2} \cap L,$$ in the viscosity sense of Definition \ref{linearsol}. 

Assume by contradiction that there exist $\phi_i$ continuous around  $X_0$ and satisfying $$\phi_1(X) =l(x') + b_1(X_0) r + O(|x'-x'_0|^2 + r^{3/2}), $$
$$\phi_2(X) = l(x')+ b_2(X_0) r + O(|x'-x'_0|^2+r^{3/2}), $$ with $l(x')= a_0 + a_1\cdot (x'-x'_0),$ 
and $$b_1(X_0) +b_2(X_0)>0$$ with  $(\phi_1, \phi_2)$ which touch $(\tilde u^-_\infty,\tilde u^+_\infty)$ by below at $X_0$.

Then we can find constants $\alpha,  \beta, \delta, \bar r$  and a point $Y'=(y'_0,0,0) \in B_2$ depending on $\phi_1,\phi_2$ such that the polynomials

$$q_1(X)=a_0 - \frac{\alpha}{2}|x'-y'_0|^2 + 2\alpha (n-1)x_n r+ 2\alpha \beta r,$$
$$q_2(X)=a_0 - \frac{\alpha}{2}|x'-y'_0|^2 - 2\alpha (n-1)x_n r- 2\alpha \beta r,$$
 touch $\phi_1,\phi_2$ by below at $X_0$  in a tubular neighborhood  $N_{\bar r}= \{|x'-x'_0|\leq \bar r, r \leq \bar r\}$ of $X_0,$ with 
 
 $$\phi_i- q_i \geq \delta>0, \quad \text{on $N_{\bar r} \setminus N_{\bar r/2}$.}$$ This implies that
\begin{equation}\label{second}
\tilde u^-_\infty - q_1 \geq \delta>0, \quad \tilde u^+_\infty - q_2 \geq \delta>0 \quad \text{on $N_{\bar r} \setminus N_{\bar r/2}$,}\end{equation}
and 
\begin{equation}\label{third}
\tilde u^-_\infty (X_0)- q_1(X_0) =0, \quad \tilde u^+_\infty (X_0)- q_2(X_0) =0.\end{equation}
In particular,\begin{equation}\label{thirdprime}
|\tilde u^-_\infty (X^-_k)- q_1(X^-_k)| \rightarrow 0, \quad |\tilde u^+_\infty (X^+_k)- q_2(X^+_k)| \rightarrow 0, \quad X_k^\mp  \in N_{\bar r} \setminus P^\mp, X_k^\mp \rightarrow X_0. \end{equation}

Now, let us choose $R_k=1/(\alpha \eps_k)$ and let us define
$$w^k_{1}(X) = v_1^{R_k}(X-Y'+\eps_k a_0 e_n), \quad Y'=(y'_0,0,0),$$
$$w^k_{2}(X) = v_2^{R_k}(X-Y'+\eps_k a_0 e_n), \quad Y'=(y'_0,0,0),$$
%$$\tilde v_k(X) = \frac{1}{\eps_k} \tilde v_{R_k}(X)$$
with $v_i^R $ the functions defined in Proposition \ref{sub} for the choice of $\beta$ above. Then the $\eps_k$-domain variation of $w_1^k, w_2^k$, which we call $\tilde w_1^k, \tilde{\bar w}_2^k$,  can be easily computed from the definition
$$w_1^k(X - \eps_k \tilde w_1^k(X)e_n)=U(X), \quad w_2^k(X - \eps_k \tilde {\bar w}_2^k(X)e_n)=\bar U(X).$$
Indeed, since $U$ and $\bar U$ are  constant in the $x'$-direction, this identity is equivalent to
$$v_1^{R_k}(X-Y'+\eps_k a_0 e_n - \eps_k \tilde w_1^k(X)e_n) = U(X-Y'),$$
$$v_2^{R_k}(X-Y'+\eps_k a_0 e_n - \eps_k \tilde {\bar w}_2^k(X)e_n) = \bar U(X-Y'),$$
which in view of Proposition \ref{sub} gives us
$$\tilde v_1^{R_k}(X-Y') = \eps_k(\tilde w_1^k(X) -a_0),$$
$$\tilde {\bar v}_2^{R_k}(X-Y') = \eps_k(\tilde {\bar w}_2^k(X) -a_0).$$
From the choice of $R_k$, the formula for $q_i$ and \eqref{estvr}, we then conclude that
$$\tilde w_1^k (X) = q_1(X) + \alpha^2 O(\eps_k),$$
$$\tilde {\bar w}_2^k (X) = q_2(X) + \alpha^2 O(\eps_k),$$  and hence
\begin{equation}\label{first} |\tilde w_1^k - q_1| \leq C\eps_k, \quad   |\tilde {\bar w}_2^k - q_2| \leq C\eps_k, \quad \text{in $N_{\bar r} \setminus P^\mp$ respectively.}\end{equation}
Thus, from the uniform convergence of $A_k^\mp$ to $A_\infty^\mp$ and \eqref{second}-\eqref{first} we get that for all $k$ large enough
\begin{equation}\label{fcont}
\tilde u^1_k - \tilde w_1^k \geq \frac \delta 2 \quad \text{in $(N_{\bar r} \setminus N_{\bar r/2}) \setminus P^-,$} 
\end{equation}
\begin{equation}\label{fcont2}
\tilde u^2_k - \tilde {\bar w}_2^k \geq \frac \delta 2 \quad \text{in $(N_{\bar r} \setminus N_{\bar r/2}) \setminus P^+.$} 
\end{equation}
Similarly, from the uniform convergence of $A_k^\mp$ to $A_\infty^\mp$ and \eqref{first}-\eqref{thirdprime} we get that for $k$ large
\begin{equation}\label{scont}
\tilde u^1_k(X^-_k) - \tilde w_1^k(X^-_k) \leq \frac \delta 4,  \quad \text{ for some sequence $X^-_k \in N_{\bar r} \setminus P^-, X^-_k \rightarrow X_0.$}\end{equation} 
\begin{equation}\label{scont2}
\tilde u^2_k(X^+_k) - \tilde {\bar w}_2^k(X^+_k) \leq \frac \delta 4,  \quad \text{ for some sequence $X^+_k \in N_{\bar r} \setminus P^+, X^+_k \rightarrow X_0.$}\end{equation} 

On the other hand, it follows from Lemma \ref{linearcomp} and \eqref{fcont}-\eqref{fcont2} that
\begin{equation*}
\tilde u^1_k - \tilde w_1^k \geq \frac \delta 2, \quad  \text{in $N_{\bar r} \setminus P^-,$}\end{equation*}
\begin{equation*}
\tilde u^2_k - \tilde {\bar w}_2^k \geq \frac \delta 2, \text{in $N_{\bar r} \setminus P^+.$}\end{equation*}
We have reached a contradiction.
\end{proof}

We are now ready to present the proof of Theorem \ref{iflat}.

\

{\it Proof of Theorem $\ref{iflat}$}. Let $\rho$ be the universal constant in Lemma \ref{tildesolves} and assume by contradiction that there exists a sequence $\eps_k \to 0$ and a sequence of solutions $(u_1^k, u_2^k)$ to \eqref{FBintro} in $B_1$ which satisfies the  flatness assumption but not the conclusion of the theorem. Then by Lemmas \ref{ginfty}-\ref{tildesolves}, for $k$ large, we have in $B_{2\rho} \setminus P^\mp,$
\begin{equation}\label{1/4}  a' \cdot x' + b_- r-\frac{1}{4} \rho \leq \tilde u^k_1 (X)\leq a' \cdot x' + b_- r+ \frac{1}{4}\rho.\end{equation}
\begin{equation}  a' \cdot x' + b_+ r-\frac{1}{4} \rho \leq \tilde {\bar u}^k_2 (X)\leq a' \cdot x' + b_+ r+ \frac{1}{4}\rho.\end{equation}
Let $$\alpha= 1+ \frac 1 2 b_- \eps =1-\frac 1 2 b_+\eps.$$
Set 
$$v(X): =\alpha U(x \cdot \nu - \frac{\eps_k}{2} \rho, z), \quad \nu=(\nu', \nu_n):= \frac{(0,1) +\eps_k (a',0)}{\sqrt{1+\eps_k^2 |a'|^2}}.$$ Since $\nu_n>0,$ $v$ is strictly monotone increasing in the $e_n$-direction say in $B_{2\rho}^+(v).$
We wish to show that
\begin{equation}\label{wish}\tilde v_{\eps_k} \leq \tilde u_1^k \quad \text{in $B_{\rho+\eps_k} \setminus P^-.$}\end{equation} Then by Lemma \ref{elem} we can conclude that
$$v \leq u_1^k \quad \text{in $B_\rho$}.$$
With similar arguments we obtain that $u_1^k, u_2^k$ satisfy all the bounds in the conclusion of the theorem and reach a contradiction.

Notice that in view of Lemma \ref{basic} in the Appendix, 
$$|\tilde v_{\eps_k}| \leq C, \quad \text{$C$ universal.}$$

According to \eqref{taylor1}, we have
\begin{equation}\label{taylor}
\alpha U(t,z) = U(t + \eps_k b_- r  + O(\eps_k^2),z), \quad r^2=t^2+z^2 .
\end{equation}
Thus we can estimate $\tilde v_{\eps_k}$ using the definition:
$$v(X-\eps_k \tilde v_{\eps_k}(X)e_n)=U(X), \quad X \not \in P^-.$$
Indeed,
$$v(X-\eps \tilde v_{\eps_k}(X)e_n)= \alpha U(t(X), z), \quad t(X):=x \cdot \nu -\eps_k \tilde v_\eps(X)\nu_n - \frac{\eps_k}{2}\rho,$$
hence by the claim,
$$U(X) = U(t(X) + \eps_k b_-\bar r + O(\eps_k^2)\bar r ^2, z ), \quad \bar r^2 = t(X)^2+z^2 .$$
We conclude that, say in $B_1 \setminus P^-$
$$x_n = x' \cdot \nu' + x_n \nu_n -\eps \tilde v_{\eps_k}(X)\nu_n - \frac{\eps_k}{2}\rho + \eps_k b_- \bar r + O(\eps_k^2).$$
Moreover, 
$$\bar r= r \nu_n + O(\eps_k).$$
Indeed, since $\nu_n \geq 1-C\eps_k^2$ 
$$|\bar r - r\nu_n|\leq  |t(X) - x_n| + C\eps_k^2$$
and also $\nu'=O(\eps_k)$ and $\tilde v_{\eps_k}$ is bounded, thus
$$t(X) -x_n = x' \cdot \nu' + x_n(\nu_n-1) - \eps \tilde v_{\eps_k} \nu_n -\frac{\eps_k}{2}\rho= O(\eps_k).$$

Hence,
$$\tilde v_{\eps_k} (X) = \frac{1}{\eps_k \nu_n}[x' \cdot \nu' + (\nu_n-1)x_n -\frac{\eps_k}{2}\rho + \eps_k b_- r\nu_n + O(\eps_k^2)]$$
and by the definition of $\nu$,
$$\tilde v_{\eps_k}(x) \leq x' \cdot a' + b_- r - \frac{\rho}{2} + O(\eps_k).$$ The inequality in \eqref{wish} now follows immediately for $k$ large using \eqref{1/4}.

\section{Segregated critical profiles}\label{sec:segregated}

\sus{
In order to effectively describe the nodal set of our  segregated minimal profiles,  we need to take properly into account  the extremality features of such configurations. To this aim, we consider the following definition.

\begin{defn}[Class $\Geh$]\label{def:segr ent prof}
For an open, $z$-symmetric $\Omega\subset\R^{n+1}$, we define  the class $\Geh(\Omega)$ of the \emph{segregated critical profiles} as the set of even-in-$z$ vector valued functions $\bu=(u_1,\ldots,u_k)\in (H^1(\Omega))^k\cap \Ceh^{0,1/2}(\Omega)$, whose components are all nonnegative  and such that

 \begin{equation}\label{eq:system_u_i}-\Delta u_i=0 \qquad \qquad {\rm in }\; \Omega\setminus\Neh(u_i),\ \ i=1,\ldots,k,.\tag{G1} \end{equation}

 \begin{equation}\label{eq:segregation} 
\tilde\Omega=:\Omega\cap\R^n\times\{0\} =\bigcup_{i=1}^k \mathcal N(u_i), \quad \mathcal N^o(u_i) \cap \mathcal N^o(\sum_{j\neq i}u_j)=\emptyset\; \forall i=1,\dots,k\;,\tag{G2}
\end{equation}

 \begin{equation}\label{eq:G3} 
 \int_\Omega \left\{dY(X)\nabla \bu(X)\cdot\nabla \bu(X)-\dfrac{1}{2}\textrm{div} {Y}(X)|\nabla \bu(X)|^2\right\}d X=0
\;,\tag{G3}
\end{equation} for every  vector field $Y\in\mathcal C_0^\infty(\Omega;\R^{n+1})$ \sus{whose $z$-component is odd-in-$z$ } ({\it the odd-in-$z$ domain variation formula}). 

As customary, we say that $\bu\in \Geh_\text{loc}(\R^{n+1})$ if $\bu\in \Geh(B_R)$ for every $R>0$.

\end{defn}

\begin{rem}\label{rem:unconstrained}
It is immediate to check, using the penalized functionals \eqref{eq:penalized_functional}, the a priori estimates of Theorem \ref{thm: intro_local} and the maximal regularity result in Theorem \ref{thm: intro_limiting_prof},  that the minimizing critical profiles given by Definition \ref{def:segr_min_conf} are indeed segregated critical profiles in the sense above. A word of caution must be entered concerning the second condition in \eqref{eq:segregation}. Let indeed $\bu\not\equiv 0$ be a minimal energy configuration in a ball with all component identically vanishing at the characteristic hyperplane. Assume that $u_i\not\equiv 0$ then, as all the other components vanish on $\Re^n$, it is an \emph{unconstrained even-in-$z$ energy minimizer} and, as such, it can not identically vanish on $\{z=0\}$.
\end{rem}

}
Now we proceed to the study of the segregated critical profiles in the sense of Definition \ref{def:segr ent prof}.  Once more, we shall take advantage of domain variations, though in a more general framework. In what follows the domain variations formula \eqref{eq:G3} will play a key role in two complementary contests. On one hand, we will exploit it in order to prove Almgren's monotonicity formula; furthermore  we will show  that it implies a fundamental 
weak reflection law which will allow us to prove that every element of the class is in fact a viscosity solution in the appropriate sense.
Such a weak reflection property is expressed by the following proposition.
\begin{prop}\label{prop:wrl} Let  
$\bu=(u_1,u_2)=(a_1U,a_2\bar U)$ for some constants, $U$ being defined in \eqref{U}. Then, the validity of the domain variation formula  \eqref{eq:G3}  is equivalent to $a_1^2=a_2^2$. \end{prop} 
\begin{proof}
To see this we use the identity, which is equivalent to \eqref{eq:G3} in the specific case,
     \begin{multline}\label{eq:poho}\int_\omega \left\{dY(x)\nabla \bu(x)\cdot\nabla \bu(x)-\dfrac{1}{2}\textrm{div} {Y}(x)|\nabla \bu(x)|^2\right\}dx\\
=\int_{\partial\omega} \left\{Y(x)\cdot\nabla \bu(x)\;\nu(x)\cdot\nabla \bu(x)-\dfrac{1}{2}\nu(x)\cdot Y(x)|\nabla \bu(x)|^2\right\}d\sigma\;, 
\end{multline}
where $\omega$ is any piecewise smooth domain in $\R^{n+1}$, $z$-symmetric. Now we take $Y\equiv e_n=(0,\dots,0,1,0)$ in $\R^{n+1}$ and the cylinder 
\[
\omega=C_{l,r} = Q_l \times B_r,
\]
where, $B_r \subset \R^2$ denotes a ball of radius $r$ in the plane $(x_n,z)$,
and $Q_l \subset \R^{n-1}$ a cube of edge length equal to $2l$, both centered at zero.

We obtain, since $\nu$ is orthogonal to $e_n$ and $\nabla u_i$ on $\partial Q_l\times B_r$,
\[
\begin{split}
0&=\int_{\partial (Q_l\times B_r)} \left\{e_n\cdot\nabla \bu(x)\;\nu(x)\cdot\nabla \bu(x)-\dfrac{1}{2}\nu(x)\cdot e_n|\nabla \bu(x)|^2\right\}d\sigma\\
&=\int_{Q_l\times\partial B_r} \left\{e_n\cdot\nabla \bu(x)\;\nu(x)\cdot\nabla \bu(x)-\dfrac{1}{2}\nu(x)\cdot e_n|\nabla \bu(x)|^2\right\}d\sigma\\
&=(2l)^{n-1}\int_0^{2\pi}\left\{\dfrac{a_1^2}{4}\cos^2(\theta/2)-\dfrac {a_1^2}{8}\cos\theta-\dfrac{a_2^2}{4}\sin^2(\theta/2)-\dfrac {a_2^2}{8}\cos\theta
\right\}d\theta
\;.
\end{split}
\]

\end{proof}

In order to prove the Almgren monotonicity formula, arguing  as in \cite[Theorem 3.3]{TVZ}, we can take advantage of the local Poho\v{z}aev identity, a direct consequence of the domain variation formula  \eqref{eq:poho} for spherical domains and the particular choice of the vector field 
$Y=X-X_0$. Indeed, for $\bu \in \mathcal G(B_R((x_0,0)))$ we have
\begin{equation}\label{eqn: Pohozaev spheres for classGs}
    (1-n) \int\limits_{ B_r} \tsum_{i} |\nabla u_i|^2 \, \de x \de z + r \int\limits_{\partial B_r} \tsum_{i} |\nabla u_i|^2 \, \de \sigma = 2 r \int\limits_{\partial B_r} \tsum_{i} |\partial_{\nu} u_i|^2 \, \de \sigma
\end{equation}
for a.e. $r\in(0,R)$.

Let us define, for every $x_0 \in \R^n$,  $X_0=(x_0,0)\in\R^{n+1}$, and $r > 0$,
\[
    \begin{split}
    E(r)=E(x_0, r)=E(x_0,\bu, r) &:= \frac{1}{r^{n-1}} \int\limits_{B_r(x_0,0)} \tsum_{i}
|\nabla u_i|^2 \, \de{x} \de{z}\\
    H(r)=H(x_0,r)=H(x_0,\bu,r) &:= \frac{1}{r^{n}} \int\limits_{\partial B_r(x_0,0)}
\tsum_{i} u_i^2 \, \de{\sigma}.
    \end{split}
\]
Since $\bu \in H^1_{\loc}\left(\R^{n+1}, \R^k\right)$, both $E$ and $H$ are locally absolutely continuous functions on $(0,+\infty)$, that is, both $E'$ and $H'$ are $L^1_{\loc}(0,\infty)$ (here, $'=\de/\de r$).
\begin{thm}\label{thm:_Almgren_for_classG_s}
Let $\bu \in \mathcal G(B_{r_0}((x_0,0))$, $\bu \not\equiv 0$. For every $x_0 \in \R^n$ the function (Almgren frequency function)
\[
    N(x_0,r) := \frac{E(x_0,r)}{H(x_0, r)}
\]
is well defined on $(0,r_0)$, absolutely continuous, non decreasing, and it satisfies the identity
\begin{equation}\label{eqn: logarithmic derivative of H}
    \frac{\mathrm{d}}{\mathrm{d}r} \log H(r) = \frac{2N(r)}{r}.
\end{equation}
Moreover, if $N(r)\equiv \gamma$ on an open interval, then $N\equiv\gamma$ for every $r$, and
$\bu$ is a homogeneous function of degree $\gamma$.
\end{thm}

\begin{proof} There holds
\[
  H^2(r) N'(r) = E'(r)H(r) - E(r)H'(r)  \qquad \text{ for } r \in (r_1,r_2).
\]
The Poho\v{z}aev identity \eqref{eqn: Pohozaev spheres for classGs} can be used to compute the derivative of $E$:
\[
\begin{split}
    E'(r) &=\frac{1-n}{r^n} \int\limits_{ B_r} \tsum_{i} |\nabla u_i|^2 \, \de{x} \de{y} + \frac{1}{r^{n-1}} \int\limits_{\partial B_r} \tsum_{i} |\nabla u_i|^2 \, \de\sigma\\
     & =  \frac{2}{r^{n-1}} \int\limits_{\partial B_r} \tsum_{i} |\partial_{\nu} u_i|^2 \, \de\sigma,
\end{split}
\]
While for $H$  we find
\[
    H'(r) =  \frac{2}{r^n} \int\limits_{\partial B_r} \tsum_{i}  u_i \partial_{\nu} u_i \, \de\sigma= \frac{2}{r}E(r) .
\]
Thus the Cauchy-Schwarz inequality yields, for  $r \in (r_1,r_2)$,
\begin{equation}\label{eqn: Cauchy Schwarz for N}
\begin{split}
      \frac{r^{2n-1}}2\, & H^2(r) N'(r)=\\ 
      &=\int\limits_{\partial B_r} \tsum_{i} |\partial_{\nu} u_i|^2 \, \de \sigma  \int\limits_{\partial B_r} \tsum_{i} u_i^2 \, \de\sigma 
       - \left( \int\limits_{\partial B_r}  \tsum_{i} u_i \partial_{\nu} u_i \, \de{\sigma} \right)^2       \geq 0.
\end{split}\end{equation}
\end{proof}

To conclude this section, we observe that the monotonicity of $N(x,r)$
implies that both for $r$ small and for $r$ large the corresponding limits are well defined. Consequently we have (see, e.g. \cite[Corollary 2.6]{TT}):
\begin{prop}\label{lem: second consequence classG_s}
Let $\bu\in\classG(\Omega)$. Then
\begin{enumerate}
 \item $N(x,0^+)$ is a non negative upper semicontinuous function on $\tilde\Omega$;
 \item If $\Omega=\R^{n+1}$, then $N(x,\infty)$ is constant on $\R^n$ (possibly $\infty$).
\end{enumerate}
\end{prop}

Another relevant consequence of the monotonicity formula is the following comparison property which, with $r_2=2 r_1$, entails the doubling property and gives a unique continuation principle.
\begin{prop}\label{coro:doubling}
Given $\bu\in \Geh(\Omega)$ and $\tilde \omega\Subset \tilde\Omega$, there exist $\tilde C>0$ and $\tilde r>0$ such that
$$H(x_0,\bu,r_2)\leq H(x_0,\bu,r_1)\left(\frac{r_2}{r_1}\right)^{2\tilde C}$$ for every $x_0\in \tilde \omega$, $0<r_1<r_2\leq \tilde r$.
\end{prop}

\subsection{Compactness and convergence of blowup sequences}\label{sec:blow_up_sequences}
In what follows we will frequently happen to study the local behavior of solutions, which will be analyzed via a blowup procedure. We will therefore need a compactness theorem, in order to obtain convergence of sequences. As a main results of \cite[Proposition 6.15]{TVZ}, we have uniform bounds in spaces of H\"older functions
\begin{thm}[A priori bounds in H\"older spaces]\label{thm:_global_holder}
Let $\{\bu_{j}\}_{j}$ be a sequence of elements of $\Geh(B_R)$
such that
\[
    \| \bu_{j} \|_{L^{\infty}(B_R)} \leq \bar m,
\]
with $\bar m$ independent of $j$. Then for every $\alpha \in (0,1/2)$ there exists a constant
$C = C(\bar m,\alpha,R)$, not depending on $j$, such that
\[
    \| \bu_j\|_{\C^{0,\alpha}\left(\overline{B_{R/2}}\right)} \leq C.
\]
Furthermore, $\{\bu_j\}_{j}$ is relatively compact in $H^1(B_{R/2}) \cap \C^{0,\alpha}\left(\overline{B_{R/2}}\right)$ for every $\alpha < 1/2$.
\end{thm}

This result is the key to the study of convergence of blow-up sequences in the class $\Geh(\Omega)$. Take $\bu\in \Geh(\Omega)$ and define, for every fixed $\rho, t>0$ and $x_0\in \tilde\Omega$,  the rescaled function
	$$\bv(x)=\bv_{x_0,t}(x)=\frac{1}{\sqrt{H(x_0,t)}}\bu_{x_0,t}(x)=\frac{\bu(x_0+tx)}{{\sqrt{H(x_0,t)}}}, \qquad \qquad {\rm for }\ x\in \tilde\Omega:=\frac{\Omega-x_0}{t}.$$

It is immediate to check that $\displaystyle \bv\in \Geh(\tilde\Omega)$. Now we consider  the convergence of such blowup sequences with fixed or variable centers. Let $\omega\Subset \Omega$ and take some sequences $x_m\in \omega$, $t_m\downarrow 0$. We define a normalized blowup sequence as
$$\bu_m(x)=\bu_{x_m,t_m}=\frac{\bu(x_m+t_m x)}{{\sqrt{H(x_m,t_m)}}},\qquad\qquad  \text{for } x\in \frac{\Omega-x_m}{t_m},$$

Thanks to the a priori bounds in H\"older spaces, the blow-up sequences do possess limits that, in two different situations, have the notable property of being homogeneous functions. This property will give us the idea for the classification of nodal points.

\begin{thm}[Compactness of blowup sequences]\label{teo:blow_up_convergence}
Under the previous notations there exists a function $\bar \bu\in \Geh_\text{loc}(\R^{n+1})$ such that, up to a subsequence, $\bu_m\rightarrow \bar \bu$ in $\Ceh^{0,\alpha}_\text{loc}(\R^{n+1})$ for every $0<\alpha<1/2$ and strongly in $H^1_\text{loc}(\R^n)$. Furthermore we have  $\ \bar \bu\in \Ceh^{0,1/2}_\text{loc}(\R^{n+1})$.
\end{thm}

\begin{cor}\label{coro:U_homogeneous}
Under the previous notations, suppose that one of these situations occurs:
\begin{enumerate}
\item $x_m=x_0\in \Neh(\bu)$ for every $m$,
\item $x_m\in \Neh(\bu)$ and $x_m\rightarrow x_0 \in \Neh(\bu)$ with $N(x_0,\bu,0^+)=1/2$.
\end{enumerate}
Then there are $h\geq2$ and $\nu\geq 1/2$ such that  $\bar \bu$ is homogeneous of degree $\nu$ with $h$ nontrivial components.
%Furthermore,in case (ii)  $\nu=1/2$,  $h=2$ and the labels of the nontrivial components of $\bu$ do not depend on the selected subsequence.
\end{cor}

In other words, $\bu\in \Heh(\nu,n, h)$, the class of functions introduced in the following definition.

\begin{defn}[Segregated entire homogenous profiles]\label{def:Heh}
For $\nu>0$ and for every dimension $n$, we define the class
\[
\Heh(\nu,n,h) := \left\{\bu \in \classG_\loc(\R^{n+1}) :
\begin{array}{l}
\bu \text{ is  $\nu$-homogeneous,}\\ \text{ with $h$ nontrivial components}
\end{array}
\right\}. 
\]
%Here we say that $\bu\in \Geh_\text{loc}(\R^{n+1})$ if $\bu\in \Geh(B_R(0))$ for every $R>0$.
\end{defn}

\subsection{Segregated entire homogenous profiles }

Now we turn to the study of the entire segregated homogeneous profiles, as in Definition \ref{def:Heh}.
For such profiles, the Almgren's quotient centered at zero does not depend on $r$. The following result has been proved in \cite[Lemmata 7.6 and 7.7]{TVZ}.

\begin{lem}\label{lem:minimal}
If $h\geq 2$, we have
\[
\begin{split}
%    \nuLio_c(n) &:=\inf\{0<\nu \leq1: \exists \bv \in \classG_c \text{ non constant with } |\bv(X)| \leq
%    C(1 + |X|^{\nu}) \},\\
%    \nuLio_s(n) &:=\inf\{0<\nu \leq1: \exists \bv \in \classG_s \text{ non constant with } \bv\in
%    \C^{0,\nu}(\overline{\R^{n+1}}) \}.\\
    \nu(n,h) &= \inf\{\nu>0: \Heh(\nu,n,h) \text{ is non empty} \}=\frac12\;.
\end{split}
\]
\end{lem}

We will need to refine the characterization of this extremal frequency. Toward this aim, we will apply the improvement of flatness Theorem \ref{mainT} and thus at this point we need to connect the two definitions of solution, namely segregated critical profiles, that is elements of the class $\mathcal G$ when $k=2$, , and viscosity solutions in the sense of Definition \ref{def:viscosity2com}. 

\begin{thm}\label{teo:viscosity_k=2} Let $\bu\in\Geh(B_1)$ with $k=2$; then $\bu$ is a viscosity solution in $B_1$ in the sense of Definition \ref{def:viscosity2com}.
\end{thm}

\begin{proof} First, comparing with Definition \ref{def:segr ent prof}, (i) is trivially satisfied. As to (ii), the two conditions follow from the segregation property (G2) and the doubling Proposition \ref{coro:doubling}. Next we prove (iii). Indeed, assume not; then, there is a strict comparison subsolution $(v_1,v_2)=:\bv$ touching $u_1$ and $u_2$ by below and above respectively at  $X_0 = (x_0,0)\in F(v_1,v_2)$. With no loss of generality, we can assume that $v_1$ and $v_2$ are harmonic where not vanishing, and thus $\alpha_1(x_0)>\alpha_2(x_0)$ in Definition \ref{defsub}.

As $u_1\geq v_1$ and $F(v_1,v_2)\in\Ceh^2$, we have $H(x_0,\bu,r)\geq H(x_0,\bv,r)\geq Cr$ as $r\to 0$, thus $N(x_0,0^+)\leq 1/2$. Since by Lemma \ref{lem:minimal}, $1/2$ is the smallest possible critical frequency, we actually infer $N(x_0,0^+)= 1/2$. Now, according to Theorem \ref{teo:blow_up_convergence}, let us take a converging blowup sequence for $\bu$ with fixed centre $x_0$, such that $H(x_0,\bu,r)\simeq r$ and call its limit $\bar\bu$.  We know that the two components, say $\bar u_1$ and $\bar u_2$ are not vanishing and, up to a rotation, they are positive multiples of $(U,\bar U)$, so $(\bar u_1,\bar u_2)=(aU,a\bar U)$ for a positive constant $a$.
On the other hand, let us consider the similarly rescaled sequence 
\[\tilde\bv_m(x)=\tilde\bv_{x_0,t_m}(x)=\frac{\bv(x_0+t_m x)}{{\sqrt{H(x_0,t_m)}}}\;,
\]
and its uniform limit $\tilde \bv=(\tilde\alpha_1(x_0)U,\tilde\alpha_2(x_0)\bar U)$. As  it still touches $\bar u_1$ and $\bar u_2$ by below, and above respectively, we deduce
\[a\geq \tilde\alpha_1>\tilde\alpha_2\geq a\;,
\]
a contradiction. The proof of (iv) is symmetrical.
\end{proof}

%
%The proof is contained in the next section, where the general case of $k$-components is discussed (see Theorem \ref{teo:viscosity}) OPPURE VOGLIAMO METTERE LA DIM DEL CASO k=2 PROPRIO QUI?. 

In particular, we can apply Theorem \ref{mainT} to segregated critical profiles with two components. This will allow us to prove the following refinement of the Lemma above.

\begin{thm}\label{teo:gap_uniqueness}

Let $\bu\in\Heh(\nu,n,h)$; then we have
\begin{enumerate}
 \item (uniqueness) $\nu=1/2$ if and only if $h=2$ and, up to orthogonal transformations, $\bu=(a U,a\bar U,0,\dots,0)$, for some positive constant $a$;
 \item (gap) there is $\delta_n>0$ such that, if $h\geq 2$ and $\nu>1/2$, then $\nu>\delta_n+1/2$. \end{enumerate}\end{thm}

\begin{proof} The characterization of the homogeneous critical profiles depends on some facts related to some spectral properties on the sphere (see also \cite{TV2018,TTV2018,BB}).
In our situation, the spectral problem we need to face
takes the following form.

\begin{defn}[The spectral problem]\label{def: lambda}\label{def: gamma}
For each open subset $\omega$ of $\S^{n-1}:=\S^n\cap \R^{n}\times\{0\}$ we define the first eigenvalue
associated with $\omega$ as
\[
    \lambda_1(\omega) :=
    \inf\left\{
    \frac{\int_{\S^{n}} |\nabla_{T} u|^2 \, \de{\sigma}
    }{\int_{\S^{n}}  u^2\, \de{\sigma}} :
    u \in H^1(\S^{n}), \, \text{ even and}\; u\equiv0 \text{ on }\S^{n-1}\setminus\omega \right\}.
\]
Here $\nabla_{T} u$ stands for the (tangential) gradient of $u$ on $\S^{n}$. Furthermore,  we define
its characteristic exponent as follows
\[
\gamma(\omega)= \sqrt{ \left(\frac{n-1}{2} \right)^2 + \lambda_1(\omega)
} - \frac{n-1}{2} \;.\]
\end{defn}

As it is well known, $u$ achieves $\lambda_1(\omega)$ if and only if it is one signed, and its $ \gamma (\lambda_1(\omega))$-homogeneous extension to $\R^{n+1}$ is harmonic (see \cite{TTV2018,BB}). Furthermore, there is a unique minimizer, up to multiplicative constants, and the eigenvalue is strictly monotone with respect to inclusion. Going back to the proof of Theorem \ref{teo:gap_uniqueness} (i),  as in \cite{TV2018}, by a rearrangement argument (the foliated Schwarz symmetrization), one may prove that, among sets having the same measure, the minimal eigenvalue, and therefore characteristic exponent, is achieved if and only if $\omega$ is a spherical cap. In such a situation, writing $\Gamma(\vartheta):=\gamma(\lambda_1(\omega_\vartheta))$ for the spherical cap $\omega_\vartheta$ with opening $\vartheta$,  we have
\[
\Gamma\left(\frac{\pi}{2}\right)=\frac12\;.
\]
If $\hat\omega$ is the  spherical cup of measure one third of that of the sphere, we have \[\Gamma(\hat\omega)=\frac12+\hat\delta_n\;,\]
for some positive $\hat\delta_n$. If $h\geq3$, at least one of the components supports will intersect the sphere $\S^n$ in a set measuring less than one third of the measure of  the whole sphere. Consequently, there holds 
\[ 
\nu(n,h) = \inf\{\nu>0: \Heh(\nu,n,h) \text{ is non empty} \}\geq\frac12+\hat\delta_n\;, \qquad\forall h\geq 3\;,
\]
In other words, 
\[\bu\in \Heh(1/2,n,h)\Longrightarrow h=2\;.
\]
Now, let $\bu=(u_1,u_2)\in \Heh(1/2,n,2)$ and $\omega_i=\{x\in\S^{n-1}\;:\;u_i(x,0)\neq 0\}$; then necessarily the $\omega_i$'s have the same measure and are two complementary half spheres. Thus 
$\bu=(u_1,u_2)=(a_1U,a_2\bar U)$ for some positive constants. On the other hand, by Proposition \ref{prop:wrl} the validity of the domain variation formula  \eqref{eq:G3}  yields $a_1=a_2$.  This proves (i).

In order to prove (ii), in light of the previous discussion we only need to examine the case of two components. Assume by contradiction that there exist sequences $\nu_m\downarrow1/2$ and $\bu_m\in\Heh(\nu_m,n,2)$. Thanks to the global compactness Theorem \ref{thm:_global_holder}, we can normalize in such a way that, up to a subsequence, $\bu_m=(u_{1,m},u_{2,m})\rightarrow \bar \bu=(U,\bar U)$ in $C^{0,\alpha}_\text{loc}(\R^{n+1})$ for every $0<\alpha<1/2$ and strongly in $H^1_\text{loc}(\R^n)$. Thus, in the unit ball we get that for $m$ large and $\bar \eps$ as in Theorem \ref{mainT},
\begin{equation*} \|u_{1,m} - U\|_\infty \leq \bar \eps, \quad \|u_{2,m}- \bar U\|_\infty \leq \bar \eps \;.\end{equation*} As already remarked, $(u_{1,m}, u_{2,m})$ is a viscosity solution in the sense of Definition \ref{def:viscosity2com}. In view of Theorem \ref{mainT} we infer that $F(u_{1,m},u_{2,m})$ is a $C^{1,\alpha}$ graph in $\mathcal{B}_{\frac 1 2}$ for every $\alpha \in (0,1)$ with  $C^{1,\alpha}$ norm bounded by a constant depending on $\alpha$ and $n$. 
On the other hand, the nodal set of $\bu_m$ is a cone for every $m$. We then deduce that $F(u_{1,m},u_{2,m})$ is an hyperplane and $\nu_m=1/2$ for a sufficiently large $m$.

\end{proof}

\section{Viscosity solutions to $k$-component systems}\label{sec:viscosity_k}

This section is devoted to the discussion of a suitable notion of viscosity solution to the segregated system, in the spirit of the class $\mathcal S$ defined in \cite{ctvVariational} and used in \cite{CKL2009}. 
\begin{defn} \label{def:viscosity_solution}
We say that $\bu=(u_1,u_2,\dots,u_k)$ is a $k$-vector viscosity solution in $\Omega$  if each $u_i \geq 0$ is a  continuous function in $\Omega$ which is even with respect to $z=0$ and it satisfies

\begin{equation}\label{eq:V1}
\Delta u_i = 0 \quad \textrm{in } \Omega\setminus\Neh(u_i)\;,\tag{V1}
\end{equation}

\begin{equation}\label{eq:V2}
\tilde\Omega=\bigcup_{i=1}^k \mathcal N(u_i), \quad \mathcal N^o(u_i) \cap \mathcal N^o(\sum_{j\neq i}u_j)=\emptyset\; \forall i\;,\tag{V2}
\end{equation}

\begin{equation}\label{eq:V3}
\begin{array}{l}
\textrm{For every $i\in\{1,\dots,k\}$, if $(v_1,v_2)$ is a (strict) comparison subsolution}\\ 
\textrm{then $v_1$ and $v_2$ cannot
 touch $u_i$ and $\sum_{j\neq i}u_j$ respectively by below and}\\  \textrm{ above at a point $X_0 = (x_0,0)\in F(u_i):=\p \mathcal N^o(u_i) \cap \mathcal B_1. $}\tag{V3}
\end{array}
\end{equation}

\end{defn}
\begin{rem} In the case $k=2$ the above definition coincides with Definition \ref{def:viscosity2com}.
\end{rem}

\begin{thm}\label{teo:viscosity} Let $\bu\in\Geh(\Omega)$; then $\bu$ is a $k$-vector viscosity solution in $\Omega$.
\end{thm}

\begin{proof} The proof proceeds exactly with the same arguments used in Theorem \ref{teo:viscosity_k=2}, applied repeatedly to the pairs $(u_i, \sum_{j\neq i}u_j)$. Just notice that, in the end of the blowup process we find only two nontrivial components.

\end{proof}

\subsection{Structure of the nodal set} 

Our main interest is the study of the free boundary $\Neh(\bu)=\{x\in \tilde\Omega:\ \bu(x,0)=0\}$ for every element of \sus{$\Mah(\Omega)$}.

\begin{defn}\label{def:singular_regular_set}
Given $\bu\in\Mah(\Omega)$ we define its regular and singular nodal sets respectively by
\[
\begin{split}
\Sigma_\bu&=\{x\in \Neh(\bu):\ N(x,U,0^+)=1/2\}, \quad \text{ and } \\
S_\bu&=\Neh(\bu)\setminus \Sigma_\bu=\{x\in \Neh(\bu):\ N(x,U,0^+)>1/2\}.
\end{split}
\]
\end{defn}
As a consequence of Theorem \ref{teo:gap_uniqueness} and Proposition \ref{lem: second consequence classG_s}, the first set is relatively open while the second is relatively closed in $\Neh(\bu)$.
\begin{thm}\label{teo:Hausdorff_dim_estimates}
Let \sus{$\bu\in\Mah(\Omega)$}. Then
\begin{enumerate}
\item $\Sigma_\bu$ is a locally finite collection of  hyper-surfaces of class $C^{1,\alpha}$ (for some $0<\alpha<1$). 
\item $\Hh_\text{dim}(S_\bu)\leq n-2$ for any $n\geq 2$. Moreover, for $n=2$, $S_\bu$ is a locally finite set.
\end{enumerate}
\end{thm}
{\it Proof  of Theorem \ref{teo:Hausdorff_dim_estimates} $(ii)$.} 
We apply the Federer's Reduction Principle, Theorem \ref{teo:FRP}, to the following class of functions
\[
\begin{split}
\Feh=\left\{\bu|_{\R^{n}}\in \left(L^\infty_{\rm loc}(\R^n)\right)^k:
 \text{there exists some domain } \ \Omega\subset\R^{n+1} \right.\\ \left.\text{such that } B_2(0)\Subset\Omega \text{ and } \bu|_{\Omega}\in \Mah(\Omega)\right\},
\end{split}
\]
\sus{where $\Mah$ is the class of minimizing segregated configuration of Definition \ref{def:segr_min_conf}}.
Indeed (H1) is obvious  and (H2) holds by Theorem \ref{teo:blow_up_convergence}. In order to check (H3), let us define $\Ss: \Feh\rightarrow \Ceh$ by $\Ss(\bu)=S_\bu$ (which is closed by Theorem \ref{teo:gap_uniqueness} (ii). It is immediate to check   (H3)-(i). As for (H3)-(ii), take $\bu_m,\bu\in \Feh$ as stated. Then in particular $\bu_m\rightarrow \bu$ uniformly in $B_2(0)$ and by using Theorem \ref{thm:_global_holder} it is easy to obtain strong convergence in $H^1(B_{3/2}(0))$. Suppose now that (H3)-(ii) does not hold; then there exists a sequence $x_m\in B_1(0)$ ($x_m\rightarrow x$, up to a subsequence, for some $x$) and $\bar \varepsilon>0$ such that $N(x_m,\bu_m,0^+)\geq 1/2+\delta$ and $\text{dist}(x_m,\Ss(\bu))\geq \bar\varepsilon $. But then for any small $r$ we obtain $$
N(x_m,\bu_m,r)\geq 1/2+\delta,
$$
and hence (since $N(x_m,\bu_m,r)\rightarrow N(x,\bu,r)$ in $m$ for any $r$) $N(x,\bu,0^+)\geq 1/2+\delta $, a contradiction. \sus{It is a standard fact that $\bu$ inherits the minimizing property from the approximating sequence.}

 Finally let us prove that $d:= \Hh_\text{dim}(S_\bu) \leq n-2$. \sus{At first, we notice that the full $d=n$ dimension is not allowed for energy minimizing configurations, as already explained in Remark \ref{rem:unconstrained}.  Next, let us assume that $d=n-1$: in such a case, }we would find a function $\bv$, $\nu$--homogeneous with respect  to every point  $y\in \R^{n-1}\times \{0\}$,   such that $S_\bv\neq\emptyset$. In such a case, 
 by performing a sequence of blow-ups  centered at the elements of the canonical basis of $\R^{n-1}\times \{0\}$, we obtain the existence of a  $\nu$--homogeneous element  $\bu\in\Geh$ whose  full nodal set $\Neh(\bu)$ is invariant under translations in the same linear space. But this yields $\nu=1/2$, a contradiction. 
\qed

\

{\it Proof  of Theorem \ref{teo:Hausdorff_dim_estimates} $(i)$.} 
To prove point (i) we need to establish some preliminary issues relative to flatness and separation properties of the regular part of the free boundary.  Our next Lemma \ref{lemma:flatness_condition} gives that $\Sigma_\bu$ verifies the  $(n-1)$--dimensional $(\eps,R)$--Reifenberg flatness condition for every $0<\eps<1$ and some $R=R(\eps)>0$. It is a straightforward consequence of the compactness of the blowup sequences with variable centers, Corollary \ref{coro:U_homogeneous}, and the characterization of the entire profiles at the lowest frequency $\nu=1/2$, Theorem \ref{teo:gap_uniqueness} (see the very similar statement and proof in \cite[Lemma 5.3]{TT}). Subsequently, in Lemma \ref{prop:local_separation_property}, we will show that, in a neighborhood of a regular point, only two components are not identically vanishing, and they are a viscosity solution to the two-component system in the sense of Definition \ref{def:viscosity2com}. Moreover, as argued before, we can arrange for their normalized blow-up to converge locally uniformly to $(U, \bar U)$. This concludes the proof in light of Theorem \ref{mainT}.
\qed

\begin{lem}\label{lemma:flatness_condition}
For $\tilde\omega\Subset\tilde\Omega$ and any given $0<\eps<1$, there exists $R>0$ such that for every $x\in\Sigma_\bu\cap\tilde \omega$ and $0<r<R$ there exists a hyper-plane $H=H_{x,r}$ containing $x$ such that \footnote{Here $d_\Hh(A,B):=\max\{\sup_{a\in A}\text{dist}(a,B),\sup_{b\in B}\text{dist}(A,b)\}$ denotes the Hausdorff distance. Notice that $d_\Hh(A,B)\leq \varepsilon$ if and only if $A\subseteq N_\varepsilon(B)$ and $B\subseteq N_\varepsilon (A)$, where $N_\varepsilon(\cdot)$ is the closed $\varepsilon$--neighborhood of a set: for $A\subseteq \R^n$ and $\eps>0$, $N_\varepsilon(A)=\{x\in \R^n:\ \text{dist}(x,A)\leq \varepsilon\}$. }
\begin{equation}\label{eq:flatness_condition}
d_\Hh (\Neh(\bu)\cap B_r(x),H\cap B_r(x))\leq \eps r.
\end{equation}
\end{lem}

\begin{lem}[Local Separation Property]\label{prop:local_separation_property}
Given $x_0\in \Sigma_\bu$ there exists a radius $R_0>0$ such that   $\Beh_{R_0}(x_0)\subset\tilde\Omega$ and in $\Beh_{R_0}$ there are only two nontrivial components, say $u_i$ and $u_j$, so  that  $\Beh_{R_0}(x_0)\setminus \Neh(\bu)=\Beh_{R_0}(x_0)\cap \{|\bu|>0\} $ has exactly two (possibly disconnected) components $\tilde D_i= \{u_i>0\}$, $\tilde D_j= \{u_j>0\}$. Consequently, the pair $(u_i,u_j)$ is a viscosity solution in $B_{R_0}$ in the sense of  Definition $\ref{def:viscosity2com}.$
\end{lem}

\begin{proof}
Indeed, in view of the previous lemma, for sufficiently small $\eps>0$, we can uniquely associate with any $y\in \Neh(\bu)\cap \Beh_{R_0}(x_0)$
 and $0<r<R_0-|y-x_0|$, and ordered pair  of indices $(i,j)$ so that there exist a hyper-plane $H_{y,r}$ (passing through $y$) and a unitary vector $\nu_{y,r}$ (orthogonal to $H_{y,r}$) such that
$$
\{x+t \nu_{y,r}\in B_r(y):\ x\in H_{y,r},\ t\geq \eps r\}\subset \tilde D_i,$$  $$\{x-t \nu_{y,r}\in B_r(y):\ x\in H_{y,r},\ t\geq \eps r\}\subset \tilde D_j.
$$
Now it is easy to check  that the pair $(i,j)$ depends continuously on $y$ and $r$, thus it is locally constant.
Finally, thanks to Theorem \ref{teo:viscosity} it is immediate to check that the pair $(u_i,u_j)$ is a viscosity solution, whenever so is $\bu$ and all the other components vanish.
\end{proof}

\appendix
\section{}

We collect here some know results from \cite{DR, DS}, which we used throughout the paper. The analogue claims for $\bar U$ can be easily obtained and the details are left to the reader.

First of all, we use the following straightforward properties of the function $U$:
 \begin{enumerate}
 \item $\Delta U = 0, \quad U>0 \quad \textrm{in $\R^{n+1} \setminus P.$}$ \\
 \item $U_t= \frac 1 2  r^{-1/2}\cos \frac{\theta}{2}= \dfrac {1}{2r} U$ and $U_t > 0$ in $\R^{n+1} \setminus P.$\\
 \end{enumerate}

 Since $U_t$ is positive harmonic in $\R ^2 \setminus \{(t,0), \quad t \le 0 \} $, homogenous of
degree $-1/2$ and vanishes continuously on $\{(t,0), \quad t < 0 \} $ one can see from boundary Harnack inequality (or by direct computation) that values of $U_t$ at nearby points with the same second coordinate are comparable in dyadic rings. Precisely we have
\begin{equation}\label{diadic}\frac{U_t(t_1,s)}{U_t(t_2,s)} \le C \quad \textrm{if} \quad |t_1-t_2| \le \frac 1 2 |(t_2,s)|.\end{equation}

\begin{lem}\label{basic} Let $g \in C(B_2)$, $g \geq 0$ be a harmonic function in $B_2^+(g)$ and let $\bar X = \frac 3 2 e_n.$ Assume that 
$$g \geq U \quad \text{in $B_2$}, \quad g(\bar X) - U(\bar X) \geq \delta_0$$ for some $\delta_0>0$, then
\begin{equation}\label{gU}g \geq (1+c\delta_0) U \quad \text{in $B_{1}$}\end{equation} for a small universal constant $c$.
%\item  If $g(\bar X) - U(\bar X) \leq \delta_1$ for some $\delta_1>0$, then given $0 < \bar \delta < 1$
%$$g \leq (1+C\delta_1) U \quad \text{in $B_{1} \setminus B_{\bar \delta}$},$$ with  $C> 0$ a large constant depending on $\bar \delta$. 
In particular, for any $0 < \eps < 2$  
\begin{equation}\label{cor1}U(X + \eps e_n) \geq (1+c\eps)U(X) \quad \text{in $B_1$},\end{equation} with $c$ small universal.
\end{lem}

%\begin{cor} \label{cor1}There exists a universal constant $c>0$ such that for any $\eps>0$  
%$$U(X + \eps e_n) \geq (1+c\eps)U(X) \quad \text{in $B_1$}.$$
%\end{cor}

\begin{lem} \label{cor2} For any $\eps >0 $ small, given $2\eps < \bar \delta <1$, there exists a constant $C>0$ depending on $\bar \delta$ such that   
$$U(t + \eps, z) \leq (1+C\eps)U(t,z) \quad \text{in $\overline{B}_1 \setminus B_{\bar \delta} \subset \R^2$}.$$
\end{lem}

 \begin{lem}\label{hyp} Let $g \in C(\overline{B}_2)$, $g \geq 0$ be a harmonic function in $B_2^+(g)$ satisfying
\begin{equation}\label{Linfty}\|g- U\|_{L^\infty(\overline{B}_2)} \leq \delta, \end{equation} and 
\begin{equation}\label{flat_fb} \{x \in \mathcal{B}_2 : x_n \leq -\delta\} \subset  \{x \in \mathcal{B}_2 : g(x,0)=0\} \subset \{x \in \mathcal{B}_2 : x_n \leq \delta\},\end{equation} with $\delta >0$ small universal. Then 
\begin{equation}\label{flat2}U(X-\eps e_n) \leq g(X) \leq  U(X+\eps e_n) \quad \text{in $B_{1},$} \end{equation} for some $\eps = K \delta,$ $K$ universal.
\end{lem}

\begin{lem}\label{rk} Let $w_1, w_2 \in C(B_1)$ satisfy
$$\Delta (U_n w_i) = 0, \quad \text{in $B_1 \setminus P^-$, i=1,2.}$$ Then $w_1$ and $w_2$ cannot touch  (either by above or below) on $P^- \setminus L$, unless they coincide.
\end{lem}

\section{}
Here we state Federer's Reduction Principle in a form suited to our needs. Take a class of functions $\mathcal{F}$ invariant under rescaling and translation, and consider a map $\Seh$ which associates to each function $\Phi \in \mathcal{F}$ a subset of $\R^n$. For us $\Seh(\Phi)$ will be the singular set associated with the segregated configuration $\Phi$. This principle establishes rather elementary conditions on $\mathcal{F}$ and $\Seh$ which imply that,  the Hausdorff dimension of $\Seh(\Phi)$ for every $\Phi\in \mathcal{F}$ can be controlled by the Hausdorff dimension of $\Seh(\Phi)$ for elements which are homogeneous of some degree.

\begin{thm}[Federer's Reduction Principle]\label{teo:FRP}
Let $\mathcal{F}\subset L^{\infty}_\text{loc}(\R^n)$, and define, for any given $\bu\in\mathcal{F},\ x_0\in\R^n$ and $t>0$, the rescaled and translated function $$\bu_{x_0,t}:=\bu(x_0+t\cdot).$$
We say that $\bu_m\rightarrow \bu$ in $\mathcal{F}$ iff $\bu_m\rightarrow \bu$ in $L^{\infty}_\text{loc}(\R^n)$.

Assume that $\mathcal{F}$ satisfies the following conditions:
\begin{itemize}
\item[(H1)] (Closure under appropriate scalings and translations)
Given any $|x_0|\leq 1-t, 0<t<1$, $\rho>0$ and $\bu\in \mathcal{F}$, we have that also $\rho\cdot \bu_{x_0,t}\in \mathcal{F}$.
\item[(H2)] (Existence of a homogeneous ``blow--up'')
Given $|x_0|<1, t_m\downarrow 0$ and $\bu \in \mathcal{F}$, there exists a sequence $\rho_m\in(0,+\infty)$, a real number $\nu\geq 0$ and a  function $\bar \bu\in \mathcal{F}$ homogeneous of degree $\nu$ such that, if we define $\bu_m(x)=\bu(x_0+t_m x)/\rho_m$, then
$$\bu_m\rightarrow \bar \bu \qquad {\rm in }\ \mathcal{F},\qquad \qquad \text{ up to a subsequence}.$$
\item[(H3)] (Singular Set hypotheses)
There exists a map $\Seh:\mathcal{F}\rightarrow \Ceh$ (where $\Ceh:=\{A\subset \R^n:\ A\cap B_1(0) \text{ is relatively closed in } B_1(0)\}$) such that
\begin{itemize}
 \item [(i)]  Given $|x_0|\leq 1-t$, $0<t<1$ and $\rho>0$, it holds $$\Seh(\rho\cdot \bu_{x_0,t})=(\Seh(\bu))_{x_0,t}:=\frac{\Seh(\bu)-x_0}{t}.$$
 \item[(ii)] Given $|x_0|<1$, $t_m\downarrow 0$ and $\bu,\bar \bu\in \Feh$ such that there exists $\rho_m>0$ satisfying $\bu_m:=\rho_m \bu_{x_0,t_m}\rightarrow \bar \bu$ in $\Feh$, the following ``continuity'' property holds:
$$\forall \varepsilon>0\ \exists k(\epsilon)>0:\ k\geq k(\varepsilon) \Rightarrow \Seh(\bu_m)\cap B_1(0)\subseteq \{x\in \R^n:\ \text{dist}(x,\Seh(\bar \bu))<\varepsilon\}.$$
\end{itemize}
\end{itemize}
Then, if we define
\begin{multline}\label{eq:FRP_definition_of_d}
d=\max \left\{ {\rm dim }\ L:\ L \text{ is a vector subspace of } \R^n \text{ and there exist } \bu\in \Feh \right.\\\text{ and }\nu\geq0
\left. \text{ such that } \Seh(\bu)\neq \emptyset \text{ and } \bu_{y,t}=t^\nu \bu\ \forall y\in L,\ t>0 \right\} ,
\end{multline}
either $\Seh(\bu)\cap B_1(0)=\emptyset$ for every $\bu\in \Feh$, or else ${\rm dim}_{\Heh}(\Seh(\bu)\cap B_1(0))\leq d$ for every $\bu\in\Feh$. Moreover in the latter case there exist a function $\bv\in \Feh$, a d-dimensional subspace $L\leq \R^n$ and a real number $\nu\geq0$ such that
\begin{equation*}\label{Psi_invariant_over_L}
\bv_{y,t}=t^\nu \bv \qquad \forall y\in L, \ t>0, \qquad \quad \text{ and } \qquad \quad \Seh(\bv)\cap B_1(0)=L\cap B_1(0).
\end{equation*}
If $d=0$ then $\Seh(\bu)\cap B_\rho(0)$ is a finite set for each $\bu\in \Feh$ and $0<\rho<1$.
\end{thm}

This is the readjusted version of the Federer principle as it appears in Simon's book  \cite[Appendix A]{Sim}. The version we present here can be seen as a particular case of a generalization made by Chen (see \cite[Theorem 8.5]{Che} and \cite[Proposition 4.5]{Che2}).

%\section{Technical lemmas}

\bibliography{sus_bibliography}

\begin{thebibliography}{10}

\bibitem{BB}
Krzysztof Bogdan and Tomasz Byczkowski.
\newblock Potential theory for the {$\alpha$}-stable {S}chr\"odinger operator
  on bounded {L}ipschitz domains.
\newblock {\em Studia Math.}, 133(1):53--92, 1999.

\bibitem{CC}
Luis~A. Caffarelli and Xavier Cabr\'e.
\newblock {\em Fully nonlinear elliptic equations}, volume~43 of {\em American
  Mathematical Society Colloquium Publications}.
\newblock American Mathematical Society, Providence, RI, 1995.

\bibitem{CKL2009}
Luis~A. Caffarelli, Aram~L. Karakhanyan, and Fang-Hua Lin.
\newblock The geometry of solutions to a segregation problem for nondivergence
  systems.
\newblock {\em J. Fixed Point Theory Appl.}, 5(2):319--351, 2009.

\bibitem{CL2008}
Luis~A. Caffarelli and Fang-Hua Lin.
\newblock Singularly perturbed elliptic systems and multi-valued harmonic
  functions with free boundaries.
\newblock {\em J. Amer. Math. Soc.}, 21(3):847--862, 2008.

\bibitem{cs}
Luis~A. Caffarelli and Luis Silvestre.
\newblock An extension problem related to the fractional {L}aplacian.
\newblock {\em Comm. Partial Differential Equations}, 32(7-9):1245--1260, 2007.

\bibitem{CCCL2004}
Shu-Ming Chang, Chang-Shou Lin, Tai-Chia Lin, and Wen-Wei Lin.
\newblock Segregated nodal domains of two-dimensional multispecies
  {B}ose-{E}instein condensates.
\newblock {\em Phys. D}, 196(3-4):341--361, 2004.

\bibitem{Che2}
Xu-Yan Chen.
\newblock On the scaling limits at zeros of solutions of parabolic equations.
\newblock {\em J. Differential Equations}, 147(2):355--382, 1998.

\bibitem{Che}
Xu-Yan Chen.
\newblock A strong unique continuation theorem for parabolic equations.
\newblock {\em Math. Ann.}, 311(4):603--630, 1998.

\bibitem{ctvNehari}
Monica Conti, Susanna Terracini, and Gianmaria Verzini.
\newblock Nehari's problem and competing species systems.
\newblock {\em Ann. Inst. H. Poincar\'e Anal. Non Lin\'eaire}, 19(6):871--888,
  2002.

\bibitem{ctvOptimal}
Monica Conti, Susanna Terracini, and Gianmaria Verzini.
\newblock An optimal partition problem related to nonlinear eigenvalues.
\newblock {\em J. Funct. Anal.}, 198(1):160--196, 2003.

\bibitem{ctvVariational}
Monica Conti, Susanna Terracini, and Gianmaria Verzini.
\newblock A variational problem for the spatial segregation of
  reaction-diffusion systems.
\newblock {\em Indiana Univ. Math. J.}, 54(3):779--815, 2005.

\bibitem{dwz1}
E.~Norman Dancer, Kelei Wang, and Zhitao Zhang.
\newblock Uniform {H}\"older estimate for singularly perturbed parabolic
  systems of {B}ose-{E}instein condensates and competing species.
\newblock {\em J. Differential Equations}, 251(10):2737--2769, 2011.

\bibitem{dwz2}
E.~Norman Dancer, Kelei Wang, and Zhitao Zhang.
\newblock Dynamics of strongly competing systems with many species.
\newblock {\em Trans. Amer. Math. Soc.}, 364(2):961--1005, 2012.

\bibitem{dwz3}
E.~Norman Dancer, Kelei Wang, and Zhitao Zhang.
\newblock The limit equation for the {G}ross-{P}itaevskii equations and {S}.
  {T}erracini's conjecture.
\newblock {\em J. Funct. Anal.}, 262(3):1087--1131, 2012.

\bibitem{DR}
Daniela De~Silva and Jean~Michel Roquejoffre.
\newblock Regularity in a one-phase free boundary problem for the fractional
  {L}aplacian.
\newblock {\em Ann. Inst. H. Poincar\'e Anal. Non Lin\'eaire}, 29(3):335--367,
  2012.

\bibitem{DS}
Daniela De~Silva and Ovidiu Savin.
\newblock {$C^{2,\alpha}$} regularity of flat free boundaries for the thin
  one-phase problem.
\newblock {\em J. Differential Equations}, 253(8):2420--2459, 2012.

\bibitem{nttv}
Benedetta Noris, Hugo Tavares, Susanna Terracini, and Gianmaria Verzini.
\newblock Uniform {H}\"older bounds for nonlinear {S}chr\"odinger systems with
  strong competition.
\newblock {\em Comm. Pure Appl. Math.}, 63(3):267--302, 2010.

\bibitem{RTT}
Miguel Ramos, Hugo Tavares, and Susanna Terracini.
\newblock Extremality conditions and regularity of solutions to optimal
  partition problems involving {L}aplacian eigenvalues.
\newblock {\em Arch. Ration. Mech. Anal.}, 220(1):363--443, 2016.

\bibitem{Sim}
Leon Simon.
\newblock {\em Lectures on geometric measure theory}, volume~3 of {\em
  Proceedings of the Centre for Mathematical Analysis, Australian National
  University}.
\newblock Australian National University, Centre for Mathematical Analysis,
  Canberra, 1983.

\bibitem{stt2018}
Yannick Sire, Susanna Terracini, and Giorgio Tortone.
\newblock On the nodal set of solutions to degenerate or singular elliptic
  equations with an application to $s$-harmonic functions.
\newblock {\em preprint}, 2018, https://arxiv.org/abs/1808.01851.

\bibitem{TT}
Hugo Tavares and Susanna Terracini.
\newblock Regularity of the nodal set of segregated critical configurations
  under a weak reflection law.
\newblock {\em Calc. Var. Partial Differential Equations}, 45(3-4):273--317,
  2012.

\bibitem{TTV2018}
Susanna Terracini, Giorgio Tortone, and Stefano Vita.
\newblock On s-harmonic functions on cones.
\newblock {\em Anal. PDE}, 11(7):1653--1691, 2018.

\bibitem{TVZ}
Susanna Terracini, Gianmaria Verzini, and Alessandro Zilio.
\newblock Uniform {H}\"older bounds for strongly competing systems involving
  the square root of the laplacian.
\newblock {\em J. Eur. Math. Soc. (JEMS)}, 18(12):2865--2924, 2016.

\bibitem{TV2018}
Susanna Terracini and Stefano Vita.
\newblock On the asymptotic growth of positive solutions to a nonlocal elliptic
  blow-up system involving strong competition.
\newblock {\em Ann. Inst. H. Poincar\'e Anal. Non Lin\'eaire}, 35(3):831--858,
  2018.

\bibitem{VZ2014}
Gianmaria Verzini and Alessandro Zilio.
\newblock Strong competition versus fractional diffusion: the case of
  {L}otka-{V}olterra interaction.
\newblock {\em Comm. Partial Differential Equations}, 39(12):2284--2313, 2014.

\end{thebibliography}
\bibliographystyle{plain}
\end{document}